\pdfoutput=1
\documentclass[leqno]{article}
\usepackage{graphicx} 
\usepackage{mathtools}
\usepackage{comment}
\usepackage[abbrev,lite,nobysame]{amsrefs}
\usepackage[svgnames]{xcolor}
\usepackage{amsthm,amsfonts,amssymb}

\usepackage{thmtools}

\usepackage{mathrsfs}
\usepackage{bm}
\usepackage{fontawesome}
\usepackage{csquotes}
\usepackage{mathabx}
\usepackage{ulem}
\usepackage[protrusion=true, tracking=true, expansion=true, kerning=true,spacing=true]{microtype}
\usepackage{enumitem}
\usepackage{caption}
\usepackage{tikz}
\usepackage{pgfplots}
\usepgflibrary{patterns}
\usetikzlibrary{arrows,arrows.meta,calc,shapes,patterns.meta,decorations.pathmorphing,decorations.text,positioning,arrows,decorations.markings,shapes.arrows} 
\pgfplotsset{compat=1.10}

\tikzdeclarepattern{
	name=mylines,
	parameters={
		\pgfkeysvalueof{/pgf/pattern keys/size},
		\pgfkeysvalueof{/pgf/pattern keys/angle},
		\pgfkeysvalueof{/pgf/pattern keys/line width},
	},
	bounding box={
		(0,-0.5*\pgfkeysvalueof{/pgf/pattern keys/line width}) and
		(\pgfkeysvalueof{/pgf/pattern keys/size},
		0.5*\pgfkeysvalueof{/pgf/pattern keys/line width})},
	tile size={(\pgfkeysvalueof{/pgf/pattern keys/size},
		\pgfkeysvalueof{/pgf/pattern keys/size})},
	tile transformation={rotate=\pgfkeysvalueof{/pgf/pattern keys/angle}},
	defaults={
		size/.initial=5pt,
		angle/.initial=45,
		line width/.initial=.4pt,
	},
	code={
		\draw [line width=\pgfkeysvalueof{/pgf/pattern keys/line width}]
		(0,0) -- (\pgfkeysvalueof{/pgf/pattern keys/size},0);
	},
}

\theoremstyle{plain}
\newtheorem{theorem}{Theorem}[section]
\newtheorem{lemma}[theorem]{Lemma}
\newtheorem{corollary}[theorem]{Corollary}
\newtheorem{proposition}[theorem]{Proposition}
\newtheorem{definition}[theorem]{Definition}
\newtheorem{remark}[theorem]{Remark}

\usepackage[colorlinks=true,linkcolor=NavyBlue,
citecolor=NavyBlue,urlcolor=NavyBlue]{hyperref}

\hypersetup{hypertexnames=false}

\usepackage{cleveref}

\crefname{equation}{}{}

\crefname{appendix}{Appendix}{Appendices}
\Crefname{appendix}{Appendix}{Appendices}

\RenewCommandCopy{\theHtheorem}{\thetheorem}

\newcommand{\boxnote}[1]{\makebox[0mm][l]{$#1$}}

\usepackage{tocloft}

\renewcommand{\leq}{\leqslant}

\renewcommand{\ge}{\geqslant}
\renewcommand{\geq}{\geqslant}
\renewcommand{\omega}{\omegaup}
\renewcommand{\longrightarrow}{\to}

\title{Controllability and Exponential Mixing in Singular Interacting Particle Systems}
\date{}

\begin{document}\linespread{1.05}\selectfont
	
	\author{Laurent Mertz\,\footnote{Department of Mathematics, City University of Hong Kong, Kowloon, Hong Kong, China, e-mail: \href{mailto:lmertz@cityu.edu.hk}{lmertz@cityu.edu.hk}}\and
		Vahagn~Nersesyan\,\footnote{NYU-ECNU Institute of Mathematical Sciences at NYU Shanghai, 3663 Zhongshan Road North, Shanghai, 200062, China, e-mail: \href{mailto:vahagn.nersesyan@nyu.edu}{Vahagn.Nersesyan@nyu.edu}}\and
		Manuel~Rissel\,\footnote{ShanghaiTech University, Institute of Mathematical Sciences, 201210 Shanghai, China, e-mail: \href{mailto:mrissel@shanghaitech.edu.cn}{mrissel@shanghaitech.edu.cn}}
	}

	\maketitle
	
	\begin{abstract}

		This article concerns interacting particle systems with singular kernels, driven either by degenerate deterministic controls or by degenerate decomposable noise. In the deterministic setting, we establish global exact controllability and a topologically robust property called solid controllability. 
        Moreover, we prove a result that guarantees global approximate controllability with prescribed trajectories.        
        For stochastic dynamics, we obtain ergodicity and exponential mixing by utilizing coupling and recurrence mechanisms based on controllability.
        Our approach exploits the singularity and applies to a broad class of models, including Biot–Savart, Coulomb, Riesz, and Yukawa interactions, as well as heterogeneous multi‑species systems. 
	\end{abstract}
	
	{\renewcommand\thefootnote{}
		\footnotetext{\textbf{Keywords:} particle system, singular kernel, controllability, decomposable noise, ergodicity, mixing; }
		\footnotetext{\textbf{MSC2020:} Primary 93B05; Secondary 37A25, 60H10, 70F10, 76B47}
	}

	\setcounter{tocdepth}{2}
	\tableofcontents 

	\section{Introduction}\label{section:introduction}
	The objectives of this article are twofold. First, we introduce a new deterministic controllability mechanism for a general class of interacting particle systems with singular interactions. This yields robust and global controllability properties of various classical systems for which controllability was largely unknown in the existing literature. Second, under natural assumptions on Lyapunov structure and well-posedness that are not required for controllability, we apply our deterministic controllability results to prove exponential mixing for associated stochastic dynamics driven by highly degenerate noise of a general type. Our main interest lies in systems of the form
	\begin{equation}\label{equation:dynamics} 
		\dot{x}(t)
		= \mathbb{J}\nabla \mathcal{H}_{N+1}(x(t))
		+ \mathbb{B}\zeta(t),
	\end{equation}
	describing the motion of $N+1$ points in $\mathbb{R}^d$ with energy
	\[
	\mathcal{H}_{N+1}(x_1,\dots,x_{N+1})
	= \frac{1}{2}\sum_{i\neq j} g(x_i-x_j)
	+ \sum_i Q(x_i),
	\]
	where $g$ is a singular kernel, $Q$ a potential,~$\mathbb{J}$ an appropriate $(N+1)d \times (N+1)d$ matrix, and the input $\mathbb{B} \zeta$ represents either a noise or control. We allow highly degenerate inputs, for instance, acting merely on a single particle. To focus on this most degenerate setting covered here, we label the position of the particle which is directly impacted by $\mathbb{B} \zeta$ as $N+1$ and fix
	\[
		\mathbb{B} \coloneq \begin{bmatrix}
			0_{\mathbb{R}^{Nd\times Nd}} & 0_{\mathbb{R}^{Nd\times d}} \\
			0_{\mathbb{R}^{d\times Nd}} & I_d
		\end{bmatrix}.
	\]
	In addition to \eqref{equation:dynamics}, our results also apply to heterogeneous multi-species systems and interactions governed by other structures (see Sections~\ref{subsection:dc} and~\ref{section:expmixing} for the general setting). One of our main contributions is that we provide global controllability results for a range of classical systems by a general argument that requires only mild assumptions on the system. An important aspect of our probabilistic application is that, for arbitrarily large $N$, the noise is allowed to have fixed rank~$d$ independent of $N$; propagation of randomness through pairwise singular interactions produces exponential mixing rates for the entire system. The probabilistic mechanism behind this is well-understood, as described below, but relies on strong controllability properties, and we show here that singularity of the kernels is an advantage rather than an obstacle in this regard.
	
	Systems of the form \eqref{equation:dynamics} arise in the study of complex phenomena such as vortex motion in turbulent flows or the dynamics of~charged particle clouds. A classical example is the point vortex system that was investigated already by Helmholtz \cite{Helmholtz1858}, Kirchhoff \cite{Kirchhoff1876}, and Onsager \cite{EyinkSreenivasan2006}, see also \cite{Newton2001,MarchioroPulvirenti1994,Khanin1981,Donati2023,GlassMunnierSueur2018}.  Other prominent examples include Coulomb or general Riesz interactions \cite{Serfaty2026,Serfaty2018,Serfaty2020}. As a concrete application of ergodicity in particle systems, one can name the computation of equilibrium averages in molecular dynamics simulations; e.g., see \cite{CookeHerzogMattinglyMcKinleySchmidler2017,CookeSchmidler2008,HerzogMattingly2019}.
	
	The basic intuition formalized here is as follows. When pushing the distinguished control particle $x_{N+1}$ via external forcing quickly towards another one, say $x_1$, singular interactions guide $x_1$ in a certain direction, while the other particles at $x_2, \dots, x_N$ barely feel the exerted force during a short time interval. For a minimal example, consider the case of two particles $x_1$ and $x_2$ controlled by a distinguished particle $x_3$, that is
	\begin{equation}\label{equation:minimal}
		N=2, \quad \mathbb{J} = I_{3d}, \quad \mathcal{H}_3(x) = \sum_{1\leq i < j \leq 3}\mathscr{W}(x_i-x_j),
	\end{equation}
	where $\nabla \mathscr{W} = \mathscr{K}$ for some singular kernel $\mathscr{K}\colon \mathbb{R}^d\setminus\{0\}\longrightarrow\mathbb{R}^d$ behaving like $|z|^{-1-1/\alpha}z$ with a number $\alpha > 0$. When the control~$x_{N+1}$ is during the short time interval $[0, \beta \delta]$ chosen of the feedback form $x_1 - \delta^{\alpha} u$ for some unit vector $u$ and length $\beta > 0$, the particle at $x_1$ moves until time $\beta \delta$ approximately with velocity $\mathscr{K}(\delta^{\alpha} u) \approx \beta \delta^{-1} w$ in a direction~$w$ determined from~$u$ by~$\mathscr{K}$. When taking $0 < \delta \ll 1$, the particle at $x_2$ barely notices the exerted force until time $\beta\delta$ and remains at its position, meanwhile $x_1$ moves in direction $w$ by a distance independent of~$\delta$. Assuming for now that no singularity occurs when $\delta\longrightarrow0$ (which will be a key issue in our proofs), this can be written as
	\begin{equation*}
		\begin{aligned}
			x_1(\beta\delta) & = x_1(0) + \int_0^{\beta\delta} \!\! \big( \underset{\mathscr{K}(\delta^{\alpha}u)}{\underbrace{ \mathscr{K}(x_1(s)-x_{N+1}(s)}}) + \underset{\mbox{\footnotesize bounded for small $s$}}{\underbrace{\mathscr{K}(x_1(s)-x_2(s))}}  \big) \, ds \\
            & \approx x_1(0) + \beta w,\\
			x_2(\beta\delta) & \approx x_2(0) - \beta \delta \mathscr{K}(x_1(0)-x_2(0)) + \delta \mathscr{K}(x_2(0)-x_1(0)+\delta^{\alpha}u) \approx x_2(0),
		\end{aligned}
	\end{equation*}
	where $\beta > 0$ determines the distance that $x_1$ is displaced in the direction $w = w(u)$.
	
	By making these heuristics rigorous, and by establishing additional ingredients such as a control stage for $x_{N+1}$ and continuous dependence of the controls on the targets, we obtain global exact controllability, solid controllability, and approximate trajectorial controllability for general classes of singular particle interaction systems. This addresses a wide gap in the literature, as there is a lack of any such controllability results for Coulomb, Riesz, Yukawa, or more general singular interactions.
    Recently, notable progress in this direction has been made by \cite{DorszGlass2023}, which relies on the ideas of Filippov's convex integration for proving global exact controllability of the point vortex system. Our approach, which covers general singular interactions, is different from the latter one, but we re-obtain global exact controllability for the point vortex system, and in addition also solid controllability. 

	Controllability of deterministic dynamics is an essential tool in the study of mixing for randomly forced differential equations, in which the deterministic control is replaced by noise. The bridge between the two is provided by the probabilistic coupling method (e.g., see \cite[Section 3]{KS-12}), whose basic mechanism can be summarized as follows. Consider two copies of the forced system, issued from arbitrary initial data. A Lyapunov structure makes both copies recur to a common compact region, from which approximate controllability drives them into an arbitrarily small neighborhood of a distinguished point; there the dynamics are regular, and solid controllability provides a contraction property, so that at each such visit the two copies coincide with a probability bounded away from zero. Consequently, the copies couple after finitely many visits and remain equal thereafter, and quantifying this recurrence through the Lyapunov structure upgrades the almost-sure coupling to exponential convergence in the total variation distance. 
	
	Relying on controllability and the coupling argument described above, we obtain new results on the long-time behavior of singular interacting particle systems driven by degenerate decomposable noise. Contrary to classical probabilistic approaches such as those developed in \cite{AK-87,MT-93,Khasminskii-1980,DaPratoZabczyk-1996,HairerMattingly2011,Raq-19}, we use our abstract criterion established in \cite{MertzNersesyanRissel2024}, which is an extension of \cite{Sh-17} to non-compact settings; see also \cite{Raq-19}, where harmonic networks driven by white noise are studied using a controllability-based approach in the same spirit. To this end, we first establish an abstract result in \Cref{section:expmixing} stating that, under a Lyapunov assumption, the system admits a unique stationary measure, which is exponentially mixing. Next, we discuss the representative example of Riesz interactions with dissipation and driven by general unbounded decomposable noise acting on the $(N+1)$-th particle:
	\[
		\dot{x}(t) = \mathbb{J} \nabla \mathcal{H}_{N+1}(x(t)) +  \begin{bmatrix}
		0_{\mathbb{R}^{Nd\times Nd}} & 0_{\mathbb{R}^{Nd\times d}} \\
		0_{\mathbb{R}^{d\times Nd}} & I_d
		\end{bmatrix} \zeta,
	\]
	where $\zeta$ is a general decomposable noise of the form
	\begin{equation*}
		\zeta(t) = \sum_{k=1}^{\infty} \mathbb{I}_{[k-1,k)}(t)\eta_k(t-k+1),
	\end{equation*}
	for independent and identically distributed (i.i.d.) random variables $(\eta_k)_{k \in \mathbb{N}}$ with values in $L^2((0,1);\mathbb{R}^{(N+1)d})$. More precisely, under natural assumptions on Lyapunov structure and damping, we obtain exponential convergence to a unique stationary measure in the total variation distance for a large class of initial distributions. 
    
    It is worth mentioning that, instead of using our criterion in \cite{MertzNersesyanRissel2024} to deduce mixing from controllability, it might be feasible to follow, in part, a more classical route that consists of establishing hypoellipticity (via H\"ormander's bracket condition) near the diagonal and proving a positive lower bound for the transition density (for instance, see \cite{Raq-19}). The approximate controllability result obtained here would nonetheless be essential for that approach: hypoellipticity yields only local information near a point, whereas the coupling construction requires global reachability of such neighborhoods. Moreover, in our framework, solid controllability is a rather direct consequence of approximate controllability together with a continuity property of the state-to-control map, which for general systems is less technical than the hypoelliptic route. It should be mentioned that the abstract criterion from \cite{MertzNersesyanRissel2024} works even for systems that are merely smooth in some neighborhood, while for our probabilistic results we assume global smoothness away from the singular set. 
	
	In this article, we do not treat the case of white noise in detail to maintain a unified presentation. Indeed, white noise leads to a less regular control system for which \cite{MertzNersesyanRissel2024} is not directly applicable in general. However, one could obtain exponential mixing under degenerate white noise by adopting the approach from \cite{Raq-19}, combined with our approximate controllability. For the simplest case of non-degenerate white noise acting on all particle equations directly, one could obtain exponential mixing by using \cite{MT-93}. Further, let us remark that exponential mixing of Langevin dynamics with singular potentials, and white noise entering the momentum equations of all particles, has been obtained, for instance, by \cite{HerzogMattingly2019}. See also \cite{CookeHerzogMattinglyMcKinleySchmidler2017} for two-dimensional Hamiltonian systems with repulsive Lennard-Jones type potentials.

	\paragraph{Basic notation.} Given $m \in \mathbb{N}$, we denote by $B_{\mathbb{R}^m}(z, R)$ the open ball of radius $R>0$ in $\mathbb{R}^m$ and by $\overline{B}_{\mathbb{R}^m}(z, R)$ its closure. In general, the word \emph{ball} refers to non-degenerate balls: only balls of positive radius are considered. If not specified otherwise, we refer by $\operatorname{dist}(y,z)$ to the Euclidean distance of two points $z,y \in \mathbb{R}^m$ and by $\langle \cdot, \cdot \rangle$ to the inner product inducing the standard norm $|z| \coloneq \smash{\sqrt{z_1^2 + \cdots + z_m^2}}$. If $y \in \mathbb{R}^m$ and $S, O \subset \mathbb{R}^m$, the distance between $S$ and $O$ is denoted by $\operatorname{dist}(S, O) \coloneq \inf_{s\in S, o \in O} |s-o|$, while $\operatorname{dist}(y, S)$ means $\operatorname{dist}(\{y\}, S)$. Other notation will be introduced later in the text when needed.

	\paragraph{Outline of this article.}
	In the remaining part of the introduction, we state our controllability results, which are proved in \Cref{section:proofofmainresult}. The noise driven dynamics are presented in \Cref{section:expmixing}, where we state and prove our probabilistic results. An appendix provides some auxiliary results, including an essentially known sufficient condition for solid controllability.

	\subsection{Controllability problems}\label{subsection:dc} 
	\subsubsection{Configuration space}
	Let $d \geq 2$ and $\mathcal{M} \coloneq \mathbb{R}^d$. We aim to control the dynamics of $N+1$ singularly interacting particles, located at positions $x(t) = (x_1(t), \dots, x_{N+1}(t)) \in \mathcal{M}^{N+1}$ at time $t \geq 0$, by acting on $x_{N+1}$ with an external input. In particular, this involves the difficulty of avoiding a collapse
	\[
		\liminf\limits_{t \nearrow T_*}|x_j(t) - x_l(t)| = 0,
	\]
	where $x_j$ and $x_l$ are the positions of two distinguished particles, respectively, and  $T_*>0$ belongs to the time interval on which the system is observed. A collapse is called a collision provided that
	\[
		\lim_{t \nearrow T_*}|x_j(t) - x_l(t)| = 0.
	\]
	The evolution of~$x$~is only meaningful in the complement $\Delta_{N+1}^{\complement} \coloneq \mathcal{M}^{N+1} \setminus \Delta_{N+1}$ of the diagonal (singular set)
	\[
		\Delta_{N+1} \coloneq \left\{ x \in \mathcal{M}^{N+1} \, | \, \exists i,j \in \{1,\dots,N+1\}, \, i \neq j \colon x_i = x_j \right\}.
	\]
	Since $\mathcal{M} = \mathbb{R}^d$ is unbounded, a particle could in principle escape to infinity in finite time; such finite-time blow-up of the trajectories must be avoided as well, in addition to collisions. See also \cite{Donati2023,MarchioroPulvirenti1994,Newton2001} for further background. 

	\subsubsection{Governing equations}
	The particle systems considered in this work include \eqref{equation:dynamics} and are of the general form
	\begin{equation}\label{equation:particlesystem_deterministic}
		\begin{cases}
			\dot{x}_i = K_i(x_i-x_{N+1}) + F_i(x), \\
			\dot{x}_{N+1} = F_{N+1}(x) + \zeta,\\
			x(0) = x_{{\mathfrak{in}}},
		\end{cases} \boxnote{\qquad \quad \quad (1 \leq i \leq N)}
	\end{equation}
	where $\zeta$ is the external input (the control) and the interactions are governed by continuous functions
	\[
		F_1, \dots, F_{N+1}\colon \Delta_{N+1}^{\complement} \to \mathbb{R}^d, \quad K_1,\dots,K_N\colon \mathcal{M}\setminus \{0\} \longrightarrow \mathbb{R}^d.
	\]
	In addition to continuity, we assume that
	\begin{gather}\label{equation:F_i}
		F_i(x_1,\dots,x_{N+1}) = 	\widetilde{F}_i(x_1,\dots,x_N) + \widehat{F}_i(x_1,\dots,x_{N+1}),\\
			K_i(z) = f_i(|z|) \frac{A_iz}{|z|^{p_i}},\label{equation:generalformofsingularkernel}
	\end{gather}
	where 
	\begin{itemize}
		\item $F_{N+1}$ is locally Lipschitz on $\Delta_{N+1}^{\complement}$, 
		\item $\widetilde{F}_i$ is locally Lipschitz on $\Delta_N^{\complement}$,
		\item $\widehat{F}_i$ is locally Lipschitz on $\mathcal{M}^{N+1}$,
		\item $A_i$ is an invertible $d\times d$ matrix, 
		\item $f_i\colon [0, \infty) \longrightarrow \mathbb{R}$ is bounded Lipschitz with $f_i(0) \neq 0$, 
		\item $p_i > 1$
	\end{itemize}
	for all $1 \leq i \leq N$. See also \Cref{subsubsection:examples} for several well-known models that are covered by the formulation above.

	\begin{remark}
		 In \eqref{equation:particlesystem_deterministic}, writing $K_i$ and $F_i$ separately (instead of combining them into a single term) emphasizes the driving singular interaction between $x_i$ and the control particle $x_{N+1}$ for $1\leq i \leq N$. It is crucial for our controllability method that all $K_1, \dots, K_N$ are sufficiently singular. To clarify this notation, consider a system resembling a typical application we have in mind:
		\begin{equation}\label{equation:examplesystem}
			\dot{x}_i =  \sum_{1 \leq j \neq i \leq N+1} K(x_i-x_j) + \delta_{i, N+1} \zeta,
		\end{equation}
		for $1 \leq i \leq N+1$,  a kernel $K$ of the type \eqref{equation:generalformofsingularkernel}, and $\delta_{i, N+1} = 1$ if $i=N+1$ and $\delta_{i, N+1} = 0$ if $i\neq N+1$. The system \eqref{equation:examplesystem} can then be recast in the form~\eqref{equation:particlesystem_deterministic} by taking 
		\[
			K_1 = \dots  = K_N = K, \quad F_i = \sum_{1 \leq j \neq i \leq N} K(x_i-x_j), \quad F_{N+1} = \sum_{j=1}^N K(x_{N+1} - x_j)
		\]
		for $1 \leq i \leq N$. In summary, the interactions modeled by \eqref{equation:particlesystem_deterministic} are of the following types
		\begin{enumerate}[labelwidth=1.2cm,
			labelsep=0.4em,
			leftmargin=!,
			align=right]
			\item[$K_i$:] necessarily singular interactions between $x_i$ and $x_{N+1}$;
			\item[$F_i$:] general interactions among $x_1, \dots, x_N$ and non-singular interactions among $x_1, \dots, x_{N+1}$;
			\item[$F_{N+1}$:] general interactions among $x_1, \dots, x_{N+1}$.
		\end{enumerate}
	\end{remark}
	
	\subsubsection{Controllability notions}
	Let $T_{\mathfrak{f}}>0$ be arbitrarily fixed in the definitions below.

	\begin{definition}\label{definition:exactcontrollability_additivecontrol}
		The system~\eqref{equation:particlesystem_deterministic} is said to be \emph{exactly controllable} in time $T_{\mathfrak{f}}$, if for every initial state $x_{\mathfrak{in}} \in \Delta^{\complement}_{N+1}$ and every target state $x_{\mathfrak{f}} \in \Delta^{\complement}_{N+1}$
		there exists a control $\zeta \in C^0([0,T_{\mathfrak{f}}]; \mathbb{R}^d)$ such that the corresponding solution
		$x$ to~\eqref{equation:particlesystem_deterministic} is defined on $[0,T_{\mathfrak{f}}]$ and satisfies
		\(
		x(T_{\mathfrak{f}})=x_{\mathfrak{f}}.
		\)
	\end{definition}

	Motivated by the applications in \Cref{section:expmixing} on the long-time behavior of stochastic versions of \eqref{equation:particlesystem_deterministic}, we also prove a property called solid controllability. The latter ensures that targets from a suitable ball can be reached in a way stable under perturbations of the system or the controls. To make this precise, for any given $x_{\mathfrak{in}} \in \Delta^{\complement}_{N+1}$, let $\mathcal{Q}(T_{\mathfrak{f}}, x_{\mathfrak{in}}) \subset C^0([0,T_{\mathfrak{f}}];\mathbb{R}^d)$ be the set of all controls~$\zeta$ for which the solution~$x$ to the Cauchy problem~\eqref{equation:particlesystem_deterministic} is defined on $[0,T_{\mathfrak{f}}]$.  We endow $\mathcal{Q}(T_{\mathfrak{f}}, x_{\mathfrak{in}})$ with the $C^0([0,T_{\mathfrak{f}}];\mathbb{R}^d)$-topology 
	and write
	\[
		x(T_{\mathfrak{f}};x_{\mathfrak{in}}, \zeta) \coloneq x(T_{\mathfrak{f}}),
	\]
	where $x$ solves~\eqref{equation:particlesystem_deterministic} with initial state $x_{\mathfrak{in}}$ and control $\zeta \in \mathcal{Q}(T_{\mathfrak{f}}, x_{\mathfrak{in}})$.

	\begin{definition}\label{definition:solidcontrollability}
		The system \eqref{equation:particlesystem_deterministic} is \emph{solidly controllable} from $x_{\mathfrak{in}} \in \Delta^{\complement}_{N+1}$ in time $T_{\mathfrak{f}}$, if there exist $\delta>0$, a ball $B\subset \Delta_{N+1}^{\complement}$, and a compact set $\mathcal{C} \subset \mathcal{Q}(T_{\mathfrak{f}},x_{\mathfrak{in}})$ such that for any continuous map $\Phi\colon \mathcal{C} \to \mathcal{M}^{N+1}$ satisfying
		\[
		  \sup_{\zeta \in \mathcal{C}} | \Phi(\zeta) - x(T_{\mathfrak{f}};x_{\mathfrak{in}}, \zeta) | \leq \delta,
		\]
		one has $B \subset \Phi(\mathcal{C})$.
	\end{definition}

	\begin{remark}
	Contrary to \Cref{definition:exactcontrollability_additivecontrol}, which is a global controllability notion for arbitrary initial and target states in $\Delta_{N+1}^{\complement}$, solid controllability as in  \Cref{definition:solidcontrollability} provides exact controllability only to targets lying in some ball that might be close to the initial state or not. The key point is that, in general, one cannot choose this ball freely. Compared with exact controllability, solid controllability carries additional topological robustness in the sense that $B \subset \Phi(\mathcal{C})$ holds for all sufficiently small perturbations~$\Phi$ of the control-to-state operator $\zeta \mapsto x(T_{\mathfrak{f}};x_{\mathfrak{in}}, \zeta)$. See \cite{Coron-07} for more background on control theory and \cite{AKSS-07,AgrachevSarychev2008} for solid controllability and its applications.
	\end{remark}

	\subsubsection{Main results}
	We are now in the position to state our main result on the controllability of the system  \eqref{equation:particlesystem_deterministic}. The proof is given in \Cref{section:proofofmainresult}.
    
	\begin{theorem}\label{theorem:maincontroltheorem}
		The system \eqref{equation:particlesystem_deterministic} is exactly controllable in any time $T_{\mathfrak{f}} > 0$ and solidly controllable from any $x_{\mathfrak{in}} \in \Delta_{N+1}^{\complement}$ in any time $T_{\mathfrak{f}} > 0$.
	\end{theorem}

  	As a by-product, our proof of the above result provides also global approximate controllability with prescribed trajectories for the first $N$ particles on $[0, T_{\mathfrak{f}}]$. See \Cref{subsection:ap} for the details.
  	\begin{theorem}\label{theorem:maintrajectorial}
  		Let $T_{\mathfrak{f}} > 0$, $x_{\mathfrak{in}}, x_{\mathfrak{f}} \in \Delta_{N+1}^{\complement}$, and $\varepsilon > 0$. Moreover, let $\mathfrak{c} = (\mathfrak{c}_1, \dots, \mathfrak{c}_N)\in C^0([0,T_{\mathfrak{f}}]; \Delta_{N}^{\complement})$ be a curve with
  		\[
  			(\mathfrak{c}(0),x_{\mathfrak{in},N+1}) = x_{\mathfrak{in}}, \quad (\mathfrak{c}(T_{\mathfrak{f}}),x_{\mathfrak{f},N+1}) = x_{\mathfrak{f}}.
  		\]
  		There is a control $\zeta \in C^{\infty}_c((0,T_{\mathfrak{f}}); \mathbb{R}^d)$ such that the solution $x$ to \eqref{equation:particlesystem_deterministic} satisfies
  		\begin{equation*}
  			\sup\limits_{t \in [0, T_{\mathfrak{f}}]} \sum_{i = 1}^{N} |\mathfrak{c}_i(t) - x_i(t)|  + |x_{N+1}(T_\mathfrak{f}) - x_{\mathfrak{f}, N+1}| < \varepsilon.
  		\end{equation*}
  	\end{theorem}
  	
	The reason that the above theorem only allows to prescribe the trajectories for the first $N$ particles is that $x_{N+1}$ effectively serves as a control agent, hence its evolution cannot be prescribed freely. We note that the method from \cite{DorszGlass2023} seems to provide a related tracking property for the particular case of the point vortex system.

	\subsection{Examples of controlled particle systems}\label{subsubsection:examples}
    Our abstract setting includes the following physically relevant interaction kernels: Biot-Savart, Riesz, Coulomb, and Yukawa. In what follows, we introduce them as examples of the formulation \cref{equation:particlesystem_deterministic,equation:F_i,equation:generalformofsingularkernel}.  Further examples include multi-species systems combining these models and we allow intensity heterogeneity through coefficients $\gamma_1,\dots, \gamma_{N+1}\in\mathbb{R}\setminus\{0\}$.

\paragraph{Biot--Savart kernel (point vortices) and generalizations.}
For this interaction kernel, we restrict to the two-dimensional case $d=2$, since point vortex dynamics are intrinsically planar. The point vortex system can be derived as a limiting model when the initial vorticity of an incompressible inviscid fluid approximates a sum of Dirac masses (e.g., see \cite{Donati2023,MarchioroPulvirenti1994,Newton2001}). Using the notation from \eqref{equation:dynamics}, the uncontrolled dynamics are governed by the energy
\begin{equation}\label{eq:energyPVS}
	\mathcal{H}_{N+1}(x_1,\dots,x_{N+1}) =  - \frac{1}{2}\sum_{1 \leq j \neq i \leq N+1}
	\gamma_i \gamma_j \log |x_i-x_j|
\end{equation}
and matrix
\begin{equation*}
	\mathbb{J} = \bigoplus_{i=1}^{N+1}
	\begin{pmatrix}
		0 & \gamma_i^{-1}\\
		-\gamma_i^{-1} & 0
	\end{pmatrix}.
\end{equation*}
 This provides an important yet simple-to-state example covered by the general
formulation in \cref{equation:particlesystem_deterministic,equation:F_i,equation:generalformofsingularkernel}, in which $x_{N+1}$ plays the role of a
degenerate control for the $N$ particles $x_1,x_2,\dots,x_N$.
In this example,
\[
	K_1(z) = \dots = K_N(z) = \gamma_{N+1}\,\widehat K(z), \boxnote{\qquad \qquad \quad \,\,\, (z \neq 0),}
\]
with
\[
	\widehat{K}(z)=\frac{z^\perp}{|z|^2},
\]
and
\begin{gather*}
	F_i(x) = \sum_{1 \leq j \neq i \leq N}\gamma_j\,\widehat K(x_i-x_j), \quad F_{N+1}(x) = \sum_{j=1}^N \gamma_{j}\widehat{K}(x_{N+1} - x_j), \\	
	A_i =	\begin{pmatrix}
		0 & -1\\
		1 & 0
	\end{pmatrix}, \quad
	f_i\equiv \gamma_{N+1}, \quad p_i = 2
\end{gather*}
for each $1 \leq i \leq N$, where $z^\perp = (-z_2, z_1)^T$ denotes the counterclockwise rotated version of $z = (z_1,z_2)^T \in \mathcal{M}$. At the level of the particle equations, this extends directly to \textit{generalized point vortex systems} by replacing for $\alpha \in (0, 2]$ the Biot-Savart kernel by
\[
	\widehat K_{\alpha}(z) \coloneq \frac{z^\perp}{|z|^{1+\alpha}}.
\]
In this case, one takes
\[
	F_i(x)=\sum_{1 \leq j \neq i \leq N}\gamma_j \widehat K_{\alpha}(x_i-x_j), \quad F_{N+1}(x) = \sum_{j=1}^N \gamma_{j}\widehat K_{\alpha}(x_{N+1} - x_j)
\]
while in \eqref{equation:generalformofsingularkernel} one takes $p_i=1+\alpha$ with the same antisymmetric matrix $A_i$ and scalar factor $f_i$ as in the preceding point vortex formulation.
The choices $\alpha=1$ and $\alpha=2$ correspond, respectively, to the
classical two-dimensional Euler point vortex system and to the
surface quasi-geostrophic point vortex model \cite{CobbDonatiGodard-Cadillac2025}.

\paragraph{Logarithmic kernels (2D log gas).}
Systems of particles with $2$D Coulomb interactions are of the form \eqref{equation:dynamics} with the energy defined in \eqref{eq:energyPVS} and the matrix
\begin{equation*}
	\mathbb{J} = -\bigoplus_{i=1}^{N+1} \frac{1}{\gamma_i} I_2.
\end{equation*}
Like the point vortex example, this can be written in the form \cref{equation:particlesystem_deterministic,equation:F_i,equation:generalformofsingularkernel} by using instead of $\widehat{K}$ defined there the kernel
\[
	\widehat{K}(z) = \frac{z}{|z|^2}.
\]

\paragraph{Riesz kernels.}
Given any $d \geq 2$ and $s\in(0,d)$, the Riesz interaction energy has the form
\[
	\mathcal{H}_{N+1}(x_1,\dots,x_{N+1}) =  \frac{1}{2}\sum \limits_{1 \leq j \neq i \leq N+1}
	 \frac{\gamma_i \gamma_j}{|x_i-x_j|^s},
\]
where 
\begin{equation}\label{equation:matrixRieszGas}
	\mathbb{J} = -\bigoplus_{i=1}^{N+1} \frac{1}{\gamma_i} I_d.
\end{equation}
This corresponds in \cref{equation:particlesystem_deterministic,equation:F_i,equation:generalformofsingularkernel} to the choices
\begin{gather*}
	K_i =  \gamma_{N+1} \widehat{K}, \quad \widehat{K}(z) = \frac{s z}{|z|^{s+2}},  \quad F_i(x) = \sum_{1 \leq j \neq i \leq N}\gamma_j \widehat{K}(x_i-x_j),\\
	F_{N+1}(x) = \sum_{j=1}^N \gamma_{j}\widehat{K}(x_{N+1} - x_j)
\end{gather*}
for $1 \leq i \leq N$. Note that the Coulomb potential appears as the particular case of the Riesz interaction with $s=d-2$ in spatial dimensions $d\geq 3$. For instance, see \cite{Serfaty2018}.

\paragraph{3D Yukawa (screened Coulomb) kernel.}
Let $d = 3$ and $\kappa > 0$. The Yukawa potential \cite{Spohn1991,Yukawa1935} can be viewed as a screened version of the Coulomb potential, where the exponential factor $e^{-\kappa r}$ suppresses the interaction at distances larger than the screening length $\kappa^{-1}$. Particle systems with Yukawa interactions are governed by the energy 
\[
\mathcal{H}_{N+1}(x_1,\dots,x_{N+1}) =  \frac{1}{2}\sum_{1 \leq j \neq i \leq N+1}
 \frac{\gamma_i \gamma_j e^{-\kappa |x_i-x_j|}}{|x_i-x_j|}
\]
and matrix $\mathbb{J}$ defined as in \eqref{equation:matrixRieszGas}. In this example, 
\[
	K_1(z) = \dots = K_N(z) = \gamma_{N+1}\widehat{K},
\]
where
\begin{gather*}
	\widehat{K}(z) = \left(1+ \kappa|z|\right)e^{-\kappa |z|}\frac{z}{|z|^{3}}, \quad F_i(x) = \sum_{1 \leq j \neq i \leq N} \gamma_j \widehat{K}(x_i-x_j), \\
	F_{N+1}(x) = \sum_{j=1}^N \gamma_{j}\widehat{K}(x_{N+1} - x_j).
\end{gather*}
This fits the structure of \cref{equation:particlesystem_deterministic,equation:F_i,equation:generalformofsingularkernel}.
Yukawa potentials in general dimensions do not necessarily have simple representations (they involve Bessel functions).

		\section{Deterministic controllability mechanisms}\label{section:proofofmainresult}
		This section is devoted to the proofs of Theorems~\ref{theorem:maincontroltheorem} and~\ref{theorem:maintrajectorial}. To simplify the exposition, it is assumed that $K_1 = K_2 = \dots = K_N$ in \eqref{equation:particlesystem_deterministic}, which only impacts generic constants and parameters chosen during the proofs. Thus, instead of the notations in \eqref{equation:particlesystem_deterministic} and  \eqref{equation:generalformofsingularkernel}
        with index $i$, we write now $K$, $A$, $f$, and $p$ without index.

		The arguments are organized as follows. First, approximate controllability is shown with arbitrary target for the first particle and \enquote{target = initial state} for the others, meanwhile viewing $x_{N+1}$ as a control that can be chosen. At this point, $x_1$ resembles (up to relabeling) any $x_j$ with $1\leq j\leq N$, the choice $j=1$ being made to simplify the presentation. Next, initial and target states are prescribed for $x_{N+1}$, as well, and an additive control as in \eqref{equation:particlesystem_deterministic} is used. This argument is iterated $N$ times to obtain global approximate controllability for all particles. These results are accompanied by continuity statements for locally defined maps that send targets to controls, allowing to conclude in \Cref{corollary:uacallparticles} uniform approximate controllability (UAC) in the sense of \Cref{definition:uac}. In addition, it is shown that not only the endpoints but the entire particle trajectories of $x_1, \dots, x_N$ can be prescribed approximately; this concludes \Cref{theorem:maintrajectorial}. Finally, solid controllability is inferred from UAC via \Cref{proposition:ucisc} and global exact controllability is obtained as a combination of global approximate controllability and solid controllability.

		\subsection{Approximate controllability}\label{subsection:ap}
		We recast the control system~\eqref{equation:particlesystem_deterministic} as a set of equations for~$N$ particles that are driven by $x_{N+1}$. For small $\delta \in (0,1)$, the position $x_{N+1}^{\delta}$ of the control particle is chosen to have the feedback form
		\begin{equation}\label{equation:fedbackt}
			x_{N+1}^{\delta}(t) = x_1^{\delta}(t) - \delta^{\alpha} u,
		\end{equation}
		where  $\alpha \coloneq 1/(p-1)$ is reciprocal to the order of $K$'s singularity, and the input $u \in \mathbb{S}^{d-1}$ provides the direction  along which the control particle located at $x_{N+1}^{\delta}$ pushes the particle located at $x_1^{\delta}$.

        Due to \eqref{equation:fedbackt} and the unit norm of $u$, the first particle $x^{\delta}_1$ in \eqref{equation:particlesystem_deterministic} satisfies the problem
		\begin{equation}\label{equation:auxiliarysystem_additive}
			\begin{cases}
				\dot{x}_1^{\delta}(t) =  f(\delta^{\alpha}) \delta^{-1} A u + F_1(x^{\delta}_1(t),\dots, x^{\delta}_{N}(t), x_1^{\delta}(t) - \delta^{\alpha} u),\\
				x^{\delta}_1(0) = x_{\mathfrak{in},1}
			\end{cases} 
		\end{equation}
		and $x^{\delta}_2, \dots, x^{\delta}_N$ solve
		\begin{equation}\label{equation:auxiliarysystem_additive2}
			\begin{cases}
				\dot{x}_j^{\delta}(t)  = K(x_j^{\delta}(t)-x_1^{\delta}(t) + \delta^{\alpha} u) + F_j(x^{\delta}_1(t),\dots,x^{\delta}_{N}(t), x_1^{\delta}(t) - \delta^{\alpha} u),\\
				x^{\delta}_j(0) = x_{\mathfrak{in},j},
			\end{cases}
		\end{equation}
		for $2 \leq j \leq N$.
		
		The next theorem allows us to push $x_1$ along a straight line by using a control of the form \eqref{equation:fedbackt}, while keeping the other particles close to their initial positions.  

		\begin{theorem}\label{theorem:auxiliar_approx_particle_1}
			Given any initial configuration $(x_{\mathfrak{in},1}, \dots, x_{\mathfrak{in},N}) \in \Delta_{N}^{\complement}$, target position $x_{\mathfrak{f},1} \in \mathcal{M}$ for the first particle, and $\varepsilon > 0$, there exist $u \in \mathbb{S}^{d-1}$, $\beta > 0$, and $\delta_0 \in (0,1)$ such that for all $\delta \in (0, \delta_0)$ the solution to \eqref{equation:auxiliarysystem_additive} and \eqref{equation:auxiliarysystem_additive2} is defined on the time interval $[0, \beta\delta]$ and satisfies
			\begin{equation}\label{equations:approx}
					\sup_{t \in [0, \beta\delta]} \left(|x^{\delta}_1(t) - x_{\mathfrak{in},1} - t (\beta\delta)^{-1} (x_{\mathfrak{f},1}-x_{\mathfrak{in},1})| + \sum_{i = 2}^N |x_i^{\delta}(t) - x_{\mathfrak{in},i}|\right) < \varepsilon.
			\end{equation}
		\end{theorem}

		\begin{proof}
			The idea is as follows. First, $x_{\mathfrak{f},1}$ is slightly perturbed (if necessary) so that the resulting target $\widetilde{x}_{\mathfrak{f},1}$ is connected with (but not equal to) $x_{\mathfrak{in},1}$ by a line segment that avoids the initial positions of all other particles. Steering $x_1$ through a tubular neighborhood of a straight path is the central idea of the argument. To this end, a direction $u \in \mathbb{S}^{d-1}$ and a parameter $\beta > 0$ are chosen such that the line segment $x_{\mathfrak{in},1} + t f(0) A u$ parametrized by $t \in [0,\beta]$ connects~$x_{\mathfrak{in},1}$ with~$\widetilde{x}_{\mathfrak{f},1}$. Due to the singular scaling of the kernel, placing the control particle at distance~$\delta^{\alpha}$ from~$x_1$ generates a force of order~$\delta^{-1}$, so that over the short time interval $[0,\beta\delta]$ the first particle moves an~$O(1)$ distance along the chosen segment, whereas all others experience~$O(\delta)$ perturbations. A bootstrap argument on a maximal time interval ensures that the geometric separation persists for sufficiently small~$\delta$, yielding the desired approximation.

			\begin{figure}[ht!]
				\centering
				\resizebox{0.5\textwidth}{!}{
					\begin{tikzpicture}
						\clip(0.7,1.8) rectangle (6.5,6.5);
						
						\draw[line width=0.2mm, color=SteelBlue] plot[smooth cycle] (5,5) circle (1.2);
						\draw[line width=0.2mm,color=SteelBlue] (5,5.1) -- (5,6.2);
						
						\draw[->, line width=0.5mm,color=FireBrick] (1,1+1) -- (4.92,4.92);
						\draw[->, line width=0.5mm,color=SeaGreen] (1,1+1) -- (4.52,4.92);
						
						\fill[color=black, fill=black] plot[smooth cycle] (1,1+1) circle (0.1);
						\draw[color=black] plot[smooth cycle] (5,5) circle (0.1);
						\draw[color=black, ] plot[smooth cycle] (4.6,5) circle (0.1);
						
						\fill[color=black, fill=black] plot[smooth cycle] (2,3+1) circle (0.1);
						\fill[color=black, fill=black] plot[smooth cycle] (0.8,1.8+1) circle (0.1);
						\fill[color=black, fill=black] plot[smooth cycle] (3.7,3.7+1) circle (0.1);
						\fill[color=black, fill=black] plot[smooth cycle] (5.5,2.5) circle (0.1);
						\fill[color=black, fill=black] plot[smooth cycle] (1.5,4.3+1) circle (0.1);
						\fill[color=black, fill=black] plot[smooth cycle] (2.9,2.2+1.2) circle (0.1);
						\fill[color=black, fill=black] plot[smooth cycle] (4.9,4.3) circle (0.1);

						\coordinate[label=right:\footnotesize{$x_{\mathfrak{in},1}$}] (Start) at (1,0.9+1);
						\coordinate[label=right:\footnotesize{$x_{\mathfrak{in},i}$}] (Start) at (2.85,2.2+1);
						\coordinate[label=right:\footnotesize{\color{SteelBlue}$\varepsilon$}] (Start) at (4.93,5.6);
						\coordinate[label=right:{\footnotesize$x_{\mathfrak{f},1}$}] (End) at (5.02,4.95);
						\coordinate[label=left:{\footnotesize$\widetilde{x}_{\mathfrak{f},1}$}] (End2) at (4.63,5.15);
						
						\coordinate[label={[rotate=37]left:{\color{SeaGreen}\footnotesize$\beta f(0)Au$}}] (End3) at (2.9,3.8);
						
					\end{tikzpicture}
				}
				\caption{The black points indicate the initial configuration in \Cref{theorem:auxiliar_approx_particle_1}, while the white circled points represent the original and perturbed final configuration for the first particle, respectively. The final positions for the other particles are chosen identical to their initial positions. The position $\widetilde{x}_{\mathfrak{f},1}$ is chosen in the $\varepsilon/2$-neighborhood of ${x}_{\mathfrak{f},1}$ such that ${x}_{\mathfrak{in},1}$ and $\widetilde{x}_{\mathfrak{f},1}$ lie on a straight line segment that avoids $x_{\mathfrak{in},i}$ for all $2 \leq i \leq N$. The (red) arrow, which starts at $x_{\mathfrak{in},1}$ and ends at $x_{\mathfrak{f},1}$ crosses the initial position of another particle, indicating that pushing $x_1$ in this direction arbitrarily fast could produce a collision. The (green) arrow, which starts at $x_{\mathfrak{in},1}$ and ends at $\widetilde{x}_{\mathfrak{f},1}$ indicates the vector $\beta f(0)Au$ used in the proof of \Cref{theorem:auxiliar_approx_particle_1}, where $\beta$ and $u$ are chosen as explained in the context of \eqref{equation:assumption}. } 
				\label{Figure:Line}
			\end{figure}
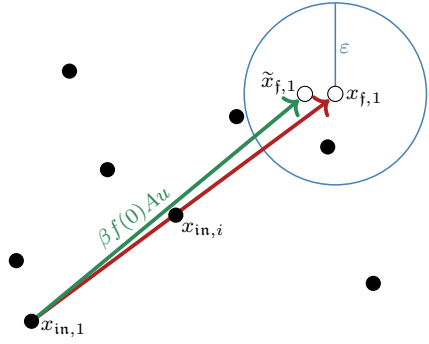

			{\it Step 1. Perturbed target.} A modified target $\widetilde{x}_{\mathfrak{f},1} \in \mathcal{M}$ is fixed such that there exist $u \in \mathbb{S}^{d-1}$ and $\beta > 0$ satisfying the geometric conditions
			\begin{equation}\label{equation:assumption}
				\begin{gathered}
					\beta f(0) A u = \widetilde{x}_{\mathfrak{f},1} - x_{\mathfrak{in},1},\\
					\forall i \in \{2,\dots,N\},  \, \forall s \in [0, \beta]\colon x_{\mathfrak{in},i} \neq x_{\mathfrak{in},1} + s f(0) A u
				\end{gathered}
			\end{equation}
			and it holds
			\begin{equation}\label{equation:mte}
				|\widetilde{x}_{\mathfrak{f},1} - x_{\mathfrak{f},1}| < \varepsilon/2.
			\end{equation}
			Thus, the line segment $\{x_{\mathfrak{in},1} + s f(0) A u\}_{s \in [0,\beta]}$ connects $x_{\mathfrak{in},1}$ with $\widetilde{x}_{\mathfrak{f},1}$ avoiding the points $x_{\mathfrak{in},2}, \dots, x_{\mathfrak{in},N}$ (see \Cref{Figure:Line}). This choice of $\widetilde{x}_{\mathfrak{f},1}$, $u$, and $\beta$ is possible due to the invertibility of $f(0) A$ and $d \geq 2$.

			{\it Step 2. Local existence time.} The Cauchy-Lipschitz theorem provides for any $\delta > 0$ a maximal time $T_{\delta} \in [0, \beta\delta]$ such that the solution $x^{\delta}$ to \eqref{equation:auxiliarysystem_additive} and \eqref{equation:auxiliarysystem_additive2} is well-defined on $[0, T_{\delta})$. If $T_{\delta} < \beta \delta$, then either a blow-up or a collapse occurs at time~$T_{\delta}$, thus one has (at least) one of the following events:
			\begin{itemize}
				\item $\liminf\limits_{t \nearrow T_{\delta}}\sum_{i=1}^N|x_i^{\delta}(t)| = +\infty$;
				\item $\exists (j, l) \in \{1, \dots, N\}^2\colon j \neq l \, \wedge \, \liminf\limits_{t \nearrow T_{\delta}}|x_j^{\delta}(t) - x_l^{\delta}(t)| = 0$;
				\item $\exists i \in \{2, \dots, N\} \colon \liminf\limits_{t \nearrow T_{\delta}}|x_i^{\delta}(t) - x_1^{\delta}(t)+\delta^{\alpha}u| = 0$.
			\end{itemize}
			Thus, proving \eqref{equations:approx} involves showing that for sufficiently small $\delta > 0$ these scenarios are impossible with $T_{\delta} \leq \beta\delta$. Without loss of generality (up to reducing~$\delta$), let $\beta \delta \leq 1$.
			
			{\it Step 3. Asymptotic expansions.} To show \eqref{equations:approx},  the following ansatz is made for the solutions $x_1^{\delta}, \dots, x_N^{\delta}$ to \eqref{equation:auxiliarysystem_additive} and \eqref{equation:auxiliarysystem_additive2}:
			\begin{equation}\label{equation:ansatz}
				\begin{cases}
					x^{\delta}_1(t) = x_{\mathfrak{in},1} + t f(\delta^{\alpha}) \delta^{-1} A u + r^{\delta}_1(t),\\
					x^{\delta}_i(t) = x_{\mathfrak{in},i} + r^{\delta}_i(t), 
				\end{cases} \boxnote{\qquad \,\,\, (2 \leq i \leq N)}
			\end{equation}
			where the remainders $r_1^{\delta}, \dots, r_N^{\delta}$ are well-defined on $[0, T_{\delta})$. 
			Inserting \eqref{equation:ansatz} into \eqref{equation:auxiliarysystem_additive} and \eqref{equation:auxiliarysystem_additive2}, one finds that~$r^{\delta}_1, \dots,r^{\delta}_N$ satisfy on $[0, T_{\delta})$ the equations
			\begin{equation}\label{equation:remainder1}
				\begin{cases}
					\begin{aligned}
						\dot{r}^{\delta}_1(t) & = F_1(x_{\mathfrak{in},1} + t f(\delta^{\alpha}) \delta^{-1} A u + r^{\delta}_1(t), x_{\mathfrak{in},2} + r^{\delta}_2(t), \dots),\\
						\dot{r}^{\delta}_i(t) & = K(x_{\mathfrak{in},i} + r^{\delta}_i(t)-x_{\mathfrak{in},1} - t f(\delta^{\alpha}) \delta^{-1} A u - r^{\delta}_1(t) + \delta^{\alpha} u) \\
						& \quad \, + F_i(x_{\mathfrak{in},1} + t f(\delta^{\alpha}) \delta^{-1} A u + r^{\delta}_1(t), x_{\mathfrak{in},2} + r^{\delta}_2(t), \dots)
					\end{aligned}
				\end{cases}
			\end{equation}
			for $2 \leq i \leq N$, and it holds
			\begin{equation}\label{equation:initialcondition_remainders}
				r^{\delta}_1(0) = \dots = r^{\delta}_N(0) = 0.
			\end{equation}
			Here, the expressions $F_i(x_{\mathfrak{in},1} + t f(\delta^{\alpha}) \delta^{-1} A u + r^{\delta}_1(t), x_{\mathfrak{in},2} + r^{\delta}_2(t), \dots)$ in \eqref{equation:remainder1} abbreviate
			\begin{equation}\label{equation:la}
				\begin{multlined}
					F_i(x_{\mathfrak{in},1} + t f(\delta^{\alpha}) \delta^{-1} A u + r^{\delta}_1(t), x_{\mathfrak{in},2} + r^{\delta}_2(t), \dots, \\
					x_{\mathfrak{in},N} + r^{\delta}_N(t), x_{\mathfrak{in},1} + t f(\delta^{\alpha}) \delta^{-1} A u + r^{\delta}_1(t) - \delta^{\alpha} u)
				\end{multlined}
			\end{equation}
			for $1 \leq i \leq N$. In particular, the argument $x_{\mathfrak{in},1} + t f(\delta^{\alpha}) \delta^{-1} A u + r^{\delta}_1(t) - \delta^{\alpha} u$ in \eqref{equation:la} enters by  \eqref{equation:F_i} only into $\widehat{F}_1, \dots, \widehat{F}_N$ and thus contributes merely a generic constant to the below estimates.
			
			{\it Step 4. Existence of $x^{\delta}$ up to time $\beta\delta$.} For the sake of applying a bootstrap argument, a uniform separation distance from the singularities in \eqref{equation:remainder1} is chosen at the initial time $t=0$. More precisely, using \eqref{equation:assumption}, two numbers $b, c \in (0,1)$ are fixed such that
			\begin{equation}\label{equation:definition_b}
				\begin{gathered}
					|x_{\mathfrak{in},l}| \leq (2b)^{-1},\\
					2b \leq |x_{\mathfrak{in},l}-x_{\mathfrak{in},j}|,\\
					2b \leq |x_{\mathfrak{in},i} - x_{\mathfrak{in},1} - \widetilde{\beta} f({\widetilde{c}}^{\alpha}) A u|, \\
					2b \leq |x_{\mathfrak{in},i} - x_{\mathfrak{in},1} - \widetilde{\beta} f({\widetilde{c}}^{\alpha}) A u + \widetilde{c}^{\alpha} u|
				\end{gathered}
			\end{equation}
			hold for all $\widetilde{c} \in (0,c)$,  $\widetilde{\beta} \in [0,\beta]$, and $1\leq i,j,l \leq N$ with $j \neq l$ and $i \neq 1$. In particular, the numbers $b$ and $c$ depend only on the fixed objects $A$,~$f$, $\beta$, ${x}_{\mathfrak{in},1}, \dots, x_{\mathfrak{in},N}$, and $x_{\mathfrak{f},1}$.
			For each $\delta \in (0,c)$, there is a maximal time  $\widetilde{T}_{\delta} \in (0, T_{\delta}]$ so that 
			\begin{equation}\label{equation:definition_tildeT}
				\begin{gathered}
					|x_{\mathfrak{in},l} + r^{\delta}_l(t)| \leq b^{-1}, \\
					b \leq |x_{\mathfrak{in},l} + r^{\delta}_l(t)-x_{\mathfrak{in},j} - r^{\delta}_j(t)|, \\
					b \leq |x_{\mathfrak{in},i} + r^{\delta}_i(t) -x_{\mathfrak{in},1} - t \delta^{-1} f(\delta^{\alpha}) A u - r^{\delta}_1(t)|, \\
					b \leq |x_{\mathfrak{in},i} + r^{\delta}_i(t) - x_{\mathfrak{in},1} - t \delta^{-1} f(\delta^{\alpha}) A u - r^{\delta}_1(t) + \delta^{\alpha} u|
				\end{gathered}
			\end{equation}
			for all $t \in [0, \widetilde{T}_{\delta}]$ and $1 \leq i,j,l \leq N$ with $j \neq l$ and $i \neq 1$.  The existence of this maximal time is ensured by \eqref{equation:initialcondition_remainders} and the continuity of $x^{\delta}_1,\dots,x^{\delta}_N$ on $[0, T_{\delta})$. In particular, the assumptions on $K$ and $F_1, \dots, F_N$ allow us to infer that
            \begin{equation}\label{KFbounds}
				\begin{aligned}
					\widetilde{C} & \coloneq \sup\limits_{\substack{2\leq i\leq N, \\ \delta \in (0,c), \\t \in [0, \widetilde{T}_{\delta}]}} |K(x_{\mathfrak{in},i} + r^{\delta}_i(t)-x_{\mathfrak{in},1} - t f(\delta^{\alpha}) \delta^{-1} A u - r^{\delta}_1(t) + \delta^{\alpha} u)|\\
					& \quad + \sup\limits_{\substack{\delta \in (0,c), \\t \in [0, \widetilde{T}_{\delta}]}} \, \sum_{i=1}^N|F_i(x_{\mathfrak{in},1} + t f(\delta^{\alpha}) \delta^{-1} A u + r^{\delta}_1(t), x_{\mathfrak{in},2} + r^{\delta}_2(t), \dots)| \\
					& < \infty.
				\end{aligned}
			\end{equation}
			
			Now, the equations in \eqref{equation:remainder1} are multiplied (in the sense of inner products) by $r^{\delta}_1, \dots, r^{\delta}_N$, respectively. Then, \cref{equation:definition_tildeT,KFbounds} yield
            \begin{equation}\label{equation:est_}
           	\begin{gathered}
				\frac{ d}{{ d}t}\sum_{i = 1}^N|r^{\delta}_i(t)|^2 \leq 2\widetilde{C} (N^{-1}B)^{-1/2}(N^{-1}B)^{1/2}\sum_{i = 1}^N|r^{\delta}_i(t)|
			\end{gathered}
            \end{equation}
			for any $B > 0$. Below, an explicit choice for $B$ will be made independently of $\delta$.
			Integrating over $[0,t]$ and  using the quadratic inequality $2lr \leq l^2+r^2$ provides for $t \in [0, \widetilde{T}_{\delta}]\subset [0,1]$ the estimate
			\begin{equation}\label{equation:bee}
				\sum_{j=1}^N |r^{\delta}_j(t)|^2 \leq q(B) \int_0^t \sum_{j=1}^N |r^{\delta}_j(s)|^2 \, ds + \frac{B}{2},
			\end{equation}
			where
			\begin{equation*}
					q(B) \coloneq 2NB^{-1} \widetilde{C}^2.
			\end{equation*}
			As $\widetilde{T}_{\delta} \in (0, \beta \delta]$, an application of Gr\"onwall's inequality subsequently yields
			\begin{equation}\label{equation:gwe} 
				\max\limits_{t \in [0, \widetilde{T}_{\delta}]}\sum_{j=1}^N |r^{\delta}_j(t)|^2 \leq \frac{B}{2} \exp\left(\delta \beta q(B)\right) \longrightarrow \frac{B}{2}
			\end{equation}
			as $\delta \longrightarrow 0$. Hence, when $\delta > 0$ is sufficiently small, it holds \smash{$\sum_{j=1}^N |r^{\delta}_j(t)|^2 < B$} for all $t \in [0, \widetilde{T}_{\delta}]$.
			
			As a consequence, there exists \smash{$\widetilde{\delta}_0 \in (0,c)$} so that for all \smash{$\delta \in (0, \widetilde{\delta}_0)$} one has \smash{$\widetilde{T}_{{\delta}} = \beta{\delta}$} and
			\begin{equation}\label{equation:gwe2}
				 \max\limits_{t \in [0, \beta\delta]}\sum_{j=1}^N |r^{\delta}_j(t)|^2 < B.
			\end{equation}
			Indeed, by contradiction assume that for any \smash{$\widetilde{\delta}_0 \in (0, c)$} there exists \smash{${\delta} \in (0, \widetilde{\delta}_0)$} with \smash{$\widetilde{T}_{{\delta}} < \beta{\delta}$}. Then, the choice \smash{$t=\widetilde{T}_{{\delta}}$} produces at least one equality sign in \eqref{equation:definition_tildeT}. Consequently, by taking in \eqref{equation:est_} the number
			\begin{equation}\label{equation:B}
				B \in \left(0, \min\left\{\frac{b^2}{4}, \frac{\varepsilon^2}{3N}\right\}\right),
			\end{equation}
			at least one of the following contradictions arises due to \cref{equation:gwe,equation:definition_b}:
				\begin{gather*}
					b^{-1} = |x_{\mathfrak{in},l} + r^{\delta}_l(\widetilde{T}_{\delta})| \leq b^{-1}/2 + \sqrt{B} < b^{-1},\\
					b = |x_{\mathfrak{in},l} + r^{\delta}_l(\widetilde{T}_{\delta})-x_{\mathfrak{in},j} - r^{\delta}_j(\widetilde{T}_{\delta})| \geq 2b- 2 \sqrt{B} > b,\\
					b = |x_{\mathfrak{in},i} + r^{\delta}_i(\widetilde{T}_{\delta}) -x_{\mathfrak{in},1} - 	\widetilde{T}_{\delta} f(\delta^{\alpha}) \delta^{-1} A u - r^{\delta}_1(\widetilde{T}_{\delta})| \geq 2b- 2 \sqrt{B } > b,\\
					b = |x_{\mathfrak{in},i} + r^{\delta}_i(\widetilde{T}_{\delta}) - x_{\mathfrak{in},1} - \widetilde{T}_{\delta} f(\delta^{\alpha}) \delta^{-1} A u - r^{\delta}_1(\widetilde{T}_{\delta}) + \delta^{\alpha} u| \geq 2b- 2 \sqrt{B} > b.
				\end{gather*}
			This shows that \eqref{equation:gwe2} holds with $B$ independent of ${\delta} \in (0, \widetilde{\delta}_0)$, and by further reducing $\widetilde{\delta}_0$ we can additionally assume that $\beta \|A\| |f({\delta}^{\alpha}) - f(0)| < B$ for all ${\delta} \in (0, \widetilde{\delta}_0)$. In particular, \cref{equation:assumption,equation:ansatz,equation:mte,equation:gwe2,equation:B} imply the existence of $\delta_0 \in (0,1)$ such that \eqref{equations:approx} holds for all $\delta \in (0,\delta_0)$.
		\end{proof}

		In the previous theorem, $x_{N+1}^{\delta}$ is treated as control input. The next goal is to study the formulation \eqref{equation:particlesystem_deterministic}, where $x_{N+1}^{\delta}$ solves itself a differential equation that is driven by an additive input $\zeta$.  
        To this end, assume that $x_{N+1}^{\delta}$ has been fixed via \Cref{theorem:auxiliar_approx_particle_1} such that $x^{\delta}_1, \dots, x^{\delta}_N$ satisfy
		\begin{equation}\label{equation:dsnp1}
			\begin{cases}
				\dot{x}^{\delta}_i = K(x^{\delta}_i-x^{\delta}_{N+1}) + F_i(x^{\delta}),\\
				\dot{x}^{\delta}_{N+1} = F_{N+1}(x^{\delta}) + \zeta^{\delta},
			\end{cases}
		\end{equation}
		where $1 \leq i \leq N$, $x^{\delta}= (x^{\delta}_1, \dots, x^{\delta}_{N+1})$, and $\zeta^{\delta} \coloneq \dot{x}^{\delta}_{N+1} -F_{N+1}(x^{\delta})$.
		
		To resolve the issue that  $x_{N+1}^{\delta}$ obtained via
		\Cref{theorem:auxiliar_approx_particle_1}  might not satisfy prescribed initial- and target conditions, a perturbation argument is employed as follows.

		\begin{theorem}\label{corollary:approx1}
			Given $x_{\mathfrak{in}} = (x_{\mathfrak{in},1}, \dots, x_{\mathfrak{in},N+1})\in \Delta_{N+1}^{\complement}$, $x_{\mathfrak{f},1}, x_{\mathfrak{f},N+1} \in \mathcal{M}$ with \smash{$(x_{\mathfrak{f},1}, x_{{\mathfrak{in}},2}, \dots, x_{{\mathfrak{in}},N}, x_{\mathfrak{f},N+1})\in \Delta_{N+1}^{\complement}$}, and $\varepsilon > 0$, there are \smash{$\beta > 0$}, $\delta_0 \in (0,1)$, and for any $\delta \in (0,\delta_0)$ a control $\zeta \in C^{0}([0,\beta\delta]; \mathbb{R}^d)$ such that the solution~$x$ to~\eqref{equation:particlesystem_deterministic}           
			exists on $[0, \beta\delta]$ and satisfies
			\begin{equation}\label{equations:approx_cor}
				\begin{gathered}
					\sup_{t \in [0, \beta\delta]} \left(|x_1(t) - x_{\mathfrak{in},1} - t (\beta\delta)^{-1} (x_{\mathfrak{f},1}-x_{\mathfrak{in},1})| + \sum_{i = 2}^N |x_i(t) - x_{\mathfrak{in},i}|\right) < \varepsilon,\\
					x_{N+1}(\beta\delta) = x_{\mathfrak{f}, N+1}, \quad \dot{x}_{N+1}(0) = \dot{x}_{N+1}(\beta \delta) = 0.
				\end{gathered}
			\end{equation}
		\end{theorem} 
		\begin{proof} 
		By \Cref{theorem:auxiliar_approx_particle_1}, there are $\beta > 0$, and $\delta_0 \in (0,1)$ so that for any $\delta \in (0, \delta_0)$ a controlled solution $x^{\delta} = (x_1^{\delta}, \dots, x_{N+1}^{\delta})$ to \eqref{equation:dsnp1}, where $x_{N+1}^{\delta}$ is determined through \cref{equation:fedbackt,equation:auxiliarysystem_additive,equation:auxiliarysystem_additive2} with initial states $x^{\delta}_1(0) = x_{\mathfrak{in},1}, \dots, x^{\delta}_{N}(0) = x_{\mathfrak{in},N}$, exists on $[0, \beta\delta]$ and satisfies
			\begin{equation}\label{equations:approx_corred}
				\begin{gathered}
					\sup_{t \in [0, \beta\delta]} \left(|x^{\delta}_1(t) - x_{\mathfrak{in},1} - t (\beta\delta)^{-1} (x_{\mathfrak{f},1}-x_{\mathfrak{in},1})| + \sum_{i = 2}^N |x_i^{\delta}(t) - x_{\mathfrak{in},i}|\right) < \varepsilon.
				\end{gathered}
			\end{equation}
			However, it is possible that $x_{N+1}^{\delta}(0)\neq x_{\mathfrak{in}, N+1}$ or  $x_{N+1}^{\delta}(\beta\delta)\neq x_{\mathfrak{f}, N+1}$. Therefore, the goal of the subsequent steps is to correct the endpoint values of $x_{N+1}^{\delta}$ without notably affecting the behavior of the first $N$ particles, which already behave as desired. During the proof, the value of $\delta_0$ might be further reduced several times (if necessary) to ensure that all constructions are well-defined.
			
			{\it Step 1. Idea of endpoint corrections.} Let $\delta \in (0, \delta_0)$. To adjust the initial and final values of $x_{N+1}^{\delta}$, we construct a family of correction profiles $(\chi_l)_{l \in \mathbb{N}} \subset C^{\infty}([0, \beta\delta]; \mathcal{M})$, satisfying
			\begin{equation}\label{equation:chillimit}
				\lim\limits_{l\to\infty}\|\chi_l\|_{L^2((0,\beta\delta); \mathcal{M})} = 0
			\end{equation}
			and
			\begin{equation}\label{equation:reqchi1}
					\begin{gathered}
						\frac{d^s \chi_l}{dt^s}(0) = \left[\frac{d^s}{dt^s}\left(x_{\mathfrak{in}, N+1} - x^{\delta}_{N+1}\right)\right](0), \quad \frac{d^k \chi_l}{dt^k}(0) = 0,\\
						\frac{d^s \chi_l}{dt^s}(\beta\delta) = \left[\frac{d^s}{dt^s}\left(x_{\mathfrak{f}, N+1} - x^{\delta}_{N+1}\right)\right](\beta\delta), \quad \frac{d^k \chi_l}{dt^k}(\beta\delta) = 0
					\end{gathered}
			\end{equation}
			for all $k \geq 2$, $s \in \{0,1\}$, and large $l \in \mathbb{N}$. The goal is then to replace $x^{\delta}_{N+1}$ by \smash{$x^{\delta}_{N+1} + \chi_l$} for large $l \in \mathbb{N}$ such that the particle positions \smash{$x^{\delta,l}_1, \dots, x^{\delta,l}_N$} associated with the new control $x^{\delta}_{N+1} + \chi_l$ behave like $x^{\delta}_1, \dots, x^{\delta}_N$ for large $l$. To justify this, we will ensure that $\chi_l(t) = 0$ for all $t \in [t_0, \beta \delta -t_0]$, where $t_0 \in (0, 2^{-1}\beta \delta)$ is fixed below, and provide a uniform Lipschitz constant $L > 0$ (independent of $l$ and $t$) with
			\begin{equation}\label{equation:tflkrtK}
				|K(z - x^{\delta}_{N+1}(t)) - K(z - x^{\delta}_{N+1}(t) - \chi_l(t))| \leq L |\chi_l(t)|
			\end{equation}
			for all pairs $(t, z)$ satisfying $t \in [0, t_0] \cup [\beta \delta-t_0, \beta \delta]$ and
			\begin{equation}\label{equation:GaGb}
				z \in \begin{cases}
					G_a = G^{\delta}_a \coloneq \bigcup_{i=1}^N \bigcup_{r \in [0, t_0]}  \overline{B}(x_i^{\delta}(r), \delta^{\alpha}/2) & \mbox{ if } t \in [0, t_0],\\
					G_b = G^{\delta}_b \coloneq \bigcup_{i=1}^N \bigcup_{r \in [\beta \delta-t_0, \beta \delta]}  \overline{B}(x_i^{\delta}(r), \delta^{\alpha}/2) & \mbox{ if } t \in [\beta \delta-t_0, \beta \delta].
				\end{cases}
			\end{equation}
			Then, assuming \eqref{equation:chillimit} and \eqref{equation:tflkrtK}, one has
			\begin{equation}\label{equation:plimitas}
				\begin{gathered}
					\lim\limits_{l\to \infty} \sup_{\substack{s,t \in [0,t_0],\\ z \in G_a}} \left| \int_s^t \left[K(z - x^{\delta}_{N+1}(r)) - K(z - x^{\delta}_{N+1}(r) - \chi_l(r))\right] \, dr \right| = 0,\\
					\lim\limits_{l\to \infty} \sup_{\substack{s,t \in [\beta\delta - t_0,\beta\delta],\\ z \in G_b}} \left| \int_s^t \left[K(z - x^{\delta}_{N+1}(r)) - K(z - x^{\delta}_{N+1}(r) - \chi_l(r))\right] \, dr \right| = 0.
				\end{gathered}
			\end{equation}
			Moreover, for the interactions described by $F_1, \dots, F_N$ one immediately obtains similar relations, noting that $\widehat{F}_1, \dots, \widehat{F}_N$ in \eqref{equation:F_i} are non-singular and in particular locally Lipschitz on $\mathcal{M}^{N+1}$.
			This will then allow (in Step 3) to compare \smash{$(x^{\delta,l}_1, \dots, x^{\delta,l}_N)$} with $(x^{\delta}_1, \dots, x^{\delta}_N)$ using the stability result \Cref{theorem:pert}.
			
			{\it Step 2. Construction of corrections.}
			By \eqref{equations:approx_corred} and \eqref{equation:GaGb}, one can reduce $\delta_0$ (if necessary) and then for any $\delta \in (0,\delta_0)$ fix $t_0 \in (0, 2^{-1}\beta \delta)$ sufficiently small such that $\mathcal{M} \setminus G_a$ and $\mathcal{M} \setminus G_b$ are connected and
			\begin{equation}\label{equation:da3}
				\begin{gathered}
					\operatorname{dist}(x_{N+1}^{\delta}([0,t_0]), G_a) \geq \delta^{\alpha}/3, \quad \operatorname{dist}(x_{N+1}^{\delta}([\beta \delta - t_0, \beta \delta]), G_b) \geq \delta^{\alpha}/3,\\
					\operatorname{dist}(x_{\mathfrak{in},N+1}, G_a) \geq \delta^{\alpha}, \quad \operatorname{dist}(x_{\mathfrak{f},N+1}, G_b) \geq \delta^{\alpha}.
				\end{gathered}
			\end{equation}
			Indeed, such a choice of~$t_0$ with~\eqref{equation:da3} is possible by combining the continuity of~$x^{\delta}$ with the following points:
			\begin{itemize}
				\item $(x_{\mathfrak{in},1}, \dots, x_{\mathfrak{in},N+1})\in \Delta_{N+1}^{\complement}$ and \smash{$(x_{\mathfrak{f},1}, x_{{\mathfrak{in}},2}, \dots, x_{{\mathfrak{in}},N}, x_{\mathfrak{f},N+1})\in \Delta_{N+1}^{\complement}$}, allow to separate $x_{\mathfrak{in},N+1}$ from $G_a$ and $x_{\mathfrak{f},N+1}$ from $G_b$;
				\item \eqref{equation:fedbackt} implies that $|x_1^{\delta}(t) - x_{N+1}^{\delta}(t)| = \delta^{\alpha}$ for all $t \in [0, \beta \delta]$;
				\item The previous points together with \eqref{equations:approx_corred} yield (reducing $\delta_0 > 0$ if necessary) that $|x_{\mathfrak{in},i}-x_{N+1}^{\delta}(0)| \geq \delta^{\alpha}$ and  $ |x_i^{\delta}(\beta \delta) - x_{N+1}^{\delta}(\beta \delta)| \geq \delta^{\alpha}/2$
				for all $\delta \in (0,\delta_0)$ and $2 \leq i \leq N$. 
			\end{itemize}
			Now, utilizing \Cref{lemma:localcurve} and the continuity of $x^{\delta}_{N+1}$, one can draw curves $\varphi^a_l, \varphi^b_l\colon [0,1]\longrightarrow \mathcal{M}$ such that
			\begin{equation*}
				\begin{gathered}
					\frac{d^s \varphi^a_l}{dt^s}(0) = l^{-s}\left[\frac{d^s}{dt^s}\left(x_{\mathfrak{in}, N+1} - x^{\delta}_{N+1}\right)\right](0), \quad \frac{d^s \varphi^a_l}{dt^s}(1) = 0, \\
					\frac{d^s \varphi^b_l}{dt^s}(0) = 0, \quad \frac{d^s \varphi^b_l}{dt^s}(1) = l^{-s}\left[\frac{d^s}{dt^s}\left(x_{\mathfrak{f}, N+1} - x^{\delta}_{N+1}\right)\right](\beta\delta),\\
					\frac{d^{k} \varphi^a_l}{dt^{k}}(0) = 0, \quad \frac{d^{k} \varphi^a_l}{dt^{k}}(1) = 0, \quad
					\frac{d^{k} \varphi^b_l}{dt^{k}}(0) = 0, \quad \frac{d^{k} \varphi^b_l}{dt^{k}}(1) = 0
				\end{gathered}
			\end{equation*}
			for $s \in \{0,1\}$, $k \geq 2$, $l \in \mathbb{N}$, and (further reducing $t_0$ if necessary)
			\begin{equation}\label{equation:beingr}
				(x^{\delta}_{N+1}(t)+\varphi^a_l(r)) \in  G_a^{\complement}, \quad (x^{\delta}_{N+1}(\beta\delta -t)+\varphi^b_l(r)) \in  G_b^{\complement}
			\end{equation}
			for all $t \in [0,t_0]$ and $r \in [0,1]$. The property \eqref{equation:beingr} ensures together with the choices of $\delta_0$ and $t_0$ the Lipschitz condition \eqref{equation:tflkrtK}, and thus yields \eqref{equation:plimitas}. 
			Finally, a family with \cref{equation:chillimit,equation:reqchi1,equation:plimitas} is defined by setting
			\[
			\chi_l(t) \coloneq \begin{cases}
				\varphi_l^a(lt) & \mbox{ if } t \in [0,1/l],\\
				\varphi_l^b(l(t - \beta \delta +1/l)) & \mbox{ if } t \in [\beta\delta - 1/l, \beta \delta],\\
				0 & \mbox{ otherwise} 
			\end{cases}
			\]
			for all integers $l >  t_0^{-1}$. 
			
			{\it Step 3. Perturbed system.} Consider the perturbed problem arising from the first~$N$ equations in \eqref{equation:dsnp1} after replacing $x^{\delta}_{N+1}$ by  $x^{\delta,l}_{N+1} \coloneq x^{\delta}_{N+1} + \chi_l$. That is,
			\begin{equation}\label{equation:dsnp1pert}
				\begin{cases}
					\dot{x}^{\delta, l}_i = K(x^{\delta, l}_i-x^{\delta,l}_{N+1}) + F_i(x^{\delta,l}_1, \dots, x^{\delta,l}_N, x^{\delta,l}_{N+1}), \\
					x^{\delta,l}_i(0) = x_{\mathfrak{in}, i},
				\end{cases}
			\end{equation}
			where $1 \leq i \leq N$. As the constructions in Step 2 ensure \eqref{equation:plimitas} and $\widehat{F}_1, \dots, \widehat{F}_N$ in~\eqref{equation:F_i} are locally Lipschitz on $\mathcal{M}^{N+1}$, the perturbative result stated in \Cref{theorem:pert} can be used to compare $(x^{\delta,l}_1, \dots, x^{\delta,l}_N)$ with the unperturbed configuration $(x^{\delta}_1, \dots, x^{\delta}_N)$.
			More precisely, fixing sufficiently large~$l=l(\delta) \in \mathbb{N}$, applying \Cref{theorem:pert} three times, and using $\chi_l = 0$ on $[t_0,\beta \delta-t_0]$, shows via \eqref{equations:approx_corred} and \eqref{equation:reqchi1} that the associated solution $(x^{\delta,l}_1, \dots, x^{\delta,l}_N)$ to~\eqref{equation:dsnp1pert} exists until $t = \beta\delta$ and satisfies	
			\begin{gather*}\label{equations:approx_corred2}
				\sup_{t \in [0, \beta\delta]} \left(|x_1^{\delta,l}(t) - x_{\mathfrak{in},1} - t (\beta\delta)^{-1} (x_{\mathfrak{f},1}-x_{\mathfrak{in},1})| +  \sum_{i = 2}^N |x_i^{\delta,l}(t) - x_{\mathfrak{in},i}|\right) < \varepsilon,\\
				x_{N+1}^{\delta,l}(\beta\delta) = x_{\mathfrak{f}, N+1}, \quad \dot{x}_{N+1}^{\delta,l}(0) = \dot{x}_{N+1}^{\delta,l}(\beta \delta) = 0.
			\end{gather*}
			Indeed, the first application of \Cref{theorem:pert} implies that $(x^{\delta,l}_1, \dots, x^{\delta,l}_N)$ converges to $(x^{\delta}_1, \dots, x^{\delta}_N)$ uniformly on $[0, t_0]$ as $l \longrightarrow \infty$. This property persists on $[t_0, \beta\delta - t_0]$ by a second application of \Cref{theorem:pert} on $[t_0, \beta\delta - t_0]$, where $\chi_l = 0$, using that \smash{$|x^{\delta,l}_i(t_0)-x^{\delta}_i(t_0)|\longrightarrow 0$} as $l \longrightarrow \infty$. The third application of \Cref{theorem:pert} accounts for the small perturbation due to $\chi_l$ on $[\beta\delta-t_0, \beta\delta]$ and the small initial perturbation at $t = \beta\delta - t_0$, using the already established closeness of \smash{$x^{\delta,l}_i(\beta\delta-t_0)$} and \smash{$x^{\delta}_i(\beta\delta-t_0)$} for large $l$. 
			
			Finally, an additive control is defined by
			\begin{equation}\label{equation:zetaformula_}
				\zeta \coloneq \dot{x}^{\delta,l}_{N+1} - F_{N+1}(x^{\delta,l}_1, \dots, x^{\delta,l}_N, x^{\delta,l}_{N+1})
			\end{equation}
			on $t \in [0,\beta\delta]$.
		\end{proof}
		
		\begin{remark}
			Using a mollification argument, one could replace $\zeta$ in the proof of \Cref{corollary:approx1} by a smooth function while maintaining the first line in \eqref{equations:approx_cor} together with $|x_{N+1}(\beta\delta) - x_{\mathfrak{f}, N+1}| < \varepsilon$.
		\end{remark}
		
		The next corollary shows that the controls constructed in Theorems~\ref{theorem:auxiliar_approx_particle_1} and~\ref{corollary:approx1} can be chosen continuous with respect to targets from certain balls. To underline that $\beta$ in this context can be bounded independently of the small parameters $\delta$ and $\varepsilon$, we denote
		\[
			\mathcal{J}_{\ell} \coloneq [\ell |f(0)|^{-1}\|A\|^{-1}, 3\ell |f(0)|^{-1}\|A^{-1}\|]
		\]
		where $\ell > 0$ and $\|A\|$ is the usual operator norm. The set $\mathcal{J}_{\ell}$ will specify the range of $\beta$ and $\ell$ will be used to associate with each particle a ball in which any target can be reached along a straight line without intersecting initial positions of other particles.
	
		\begin{corollary}\label{corollary:approx1cont}
			Given $x_{\mathfrak{in}} = (x_{\mathfrak{in},1}, \dots, x_{\mathfrak{in},N+1})\in \Delta_{N+1}^{\complement}$, let $\ell > 0$ and $\mathcal{B}_1, \dots, \mathcal{B}_{N+1}\subset \mathcal{M}$ be closed balls of diameter~$\ell$ and mutual distance at least~$4\ell$ such that
			\begin{equation}\label{equation:ell}
					\operatorname{dist}(x_{\mathfrak{in},i}, \mathcal{B}_i) \in (\ell, 2\ell) \boxnote{\qquad \qquad \quad \, (1\leq i \leq N+1).}
			\end{equation}
			For each $\varepsilon > 0$, there exists $\delta_0 > 0$ such that for all $\delta \in (0, \delta_0)$ there is a continuous map
			\[
				{\mathcal{B}}_1\times \dots \times {\mathcal{B}}_{N+1} \longrightarrow C^0([0,1];\mathbb{R}^d)\times\mathcal{J}_{\ell}, \quad x_{\mathfrak{f}} \mapsto (\zeta, \beta),
			\] 
			associating with each $x_{\mathfrak{f}} \in {\mathcal{B}}_1\times \dots \times {\mathcal{B}}_{N+1}$ a control $\zeta$ for which the solution $x$ to~\eqref{equation:particlesystem_deterministic} satisfies \eqref{equations:approx_cor}.
			
			\end{corollary}
			\begin{proof} The desired map will be defined by following the constructions described in the proofs of Theorems~\ref{theorem:auxiliar_approx_particle_1} and~\ref{corollary:approx1}. In this process, continuity with respect to $x_{\mathfrak{f},2}, \dots, x_{\mathfrak{f},N}$ is immediate, as these targets do not enter the controllability condition \eqref{equations:approx_cor}.
				
			{\it Step 1. \Cref{theorem:auxiliar_approx_particle_1}'s constructions.} By the proof of \Cref{theorem:auxiliar_approx_particle_1}, there exists $\delta_0 > 0$ such that for all $\delta \in (0,\delta_0)$ one can assign $x_{\mathfrak{f},1} \mapsto (u, \beta)$ with \eqref{equations:approx} through a continuous map ${\mathcal{B}_1}\longrightarrow \mathbb{S}^{d-1}\times\mathcal{J}_{\ell}$. Indeed, \eqref{equation:assumption} holds due to \eqref{equation:ell} with $\widetilde{x}_{\mathfrak{f},1} = x_{\mathfrak{f},1}$ so that~$u$ and~$\beta$ are determined there as continuous functions of $x_{\mathfrak{f},1}\in\mathcal{B}_1$. Moreover, the constants $b$ and $c$ in \eqref{equation:definition_b} can be taken uniformly with respect to $x_{\mathfrak{f},1}$ from bounded sets such that the estimates developed there allow choosing~$\delta_0$ independently of $x_{\mathfrak{f},1}\in\mathcal{B}_1$. The range of $\beta$ follows from \eqref{equation:assumption}. 
				
			{\it Step 2. \Cref{corollary:approx1}'s constructions.} If the value of $\delta_0$ obtained in the previous step is further reduced during the proof of \Cref{corollary:approx1}, this can be done uniformly with respect to $x_{\mathfrak{f},1}$ and $x_{\mathfrak{f},N+1}$ (from $\mathcal{B}_1\times \mathcal{B}_{N+1}$), noting that the initial positions of all particles are separated according to \eqref{equation:ell} and the minimal mutual distance of $4\ell$ between the balls $\mathcal{B}_1, \dots, \mathcal{B}_{N+1}$. Hence, suitable choices of~$t_0$ and~$l$ can be made there likewise uniformly with respect to $(x_{\mathfrak{f},1}, x_{\mathfrak{f},N+1})$. Consequently, \Cref{corollary:approx1}'s constructions provide~$\chi_l$ and~\smash{$\dot{\chi}_l$} as continuous functions of $(x_{\mathfrak{f},1}, x_{\mathfrak{f},N+1})$ with respect to the supremum norm. Indeed, given any $\delta \in (0, \delta_0)$, the endpoints in \eqref{equation:reqchi1} depend continuously on $(x_{\mathfrak{f},1}, x_{\mathfrak{f},N+1})$, and \Cref{lemma:localcurve} allows to construct $\chi_l$ and $\dot{\chi}_l$ continuously with respect to these endpoints. Notably, in the proof of  \Cref{corollary:approx1}, we apply \Cref{lemma:localcurve} only locally to attach suitable curve endings to fixed curves that connect $\mathcal{B}_{N+1}$ with each $\mathcal{B}_i$, $1 \leq i \leq N$, and these fixed connection curves are independent of $\delta, t_0$ and $l$.
			Finally, the Lipschitz constant in \eqref{equation:tflkrtK} can be taken uniformly with respect to $(x_{\mathfrak{f},1}, x_{\mathfrak{f},N+1})$ due to $\beta \geq \ell |f(0)|^{-1}\|A\|^{-1}$ and the uniform separation of the initial states in terms of $\ell$. Hence, the functions $x^{\delta,l}$ and $\dot{x}^{\delta,l}$ determined via \eqref{equation:dsnp1pert}, and also the composition $F_{N+1}(x^{\delta,l})$, are continuous in the supremum norm with respect to $(x_{\mathfrak{f},1}, x_{\mathfrak{f},N+1})$. The formula \eqref{equation:zetaformula_}, extended constantly to $[0,1]$, provides now the desired map.
			\end{proof}

		The statements of Theorems~\ref{theorem:auxiliar_approx_particle_1},~\ref{corollary:approx1}, and~\Cref{corollary:approx1cont} remain true when interchanging the roles of $x_1$ and $x_j$ for any $2 \leq j \leq N$. Thus, by iterating applications of these results at most $N$ times, each particle can be steered approximately to any prescribed target.

		\begin{corollary}\label{corollary:approx2}
				Given $x_{\mathfrak{in}}, x_{\mathfrak{f}} \in \Delta_{N+1}^{\complement}$ and $\varepsilon > 0$, there are $\beta_1, \dots, \beta_N > 0$ and $\delta_0 \in (0,1)$ such that for all $\delta \in (0, \delta_0)$ there is a control $\zeta \in C^{\infty}([0,\beta\delta]; \mathbb{R}^d)$ with $\beta\coloneq (\beta_1+\dots+\beta_N)$ such that the solution $x$ to \eqref{equation:particlesystem_deterministic} exists on $[0, \beta\delta]$ and satisfies
				\begin{equation}\label{equation:adi}
					\begin{aligned}					
						& \sum_{i =1}^N\sup_{t \in [0, t_{i-1}]} \left|x_i(t) - x_{\mathfrak{in}, i}\right| \\ 
						& +  \sum_{i=1}^N\sup_{t \in [t_{i-1}, t_i]} \left|x_i(t) - x_{i}\left(t_{i-1}\right) - \left(t-t_{i-1}\right) (\beta_i\delta)^{-1} (x_{\mathfrak{f},i}-x_{\mathfrak{in},i})\right|  \\ 
						& +  \sum_{i =1}^N\sup_{t \in [t_i, \beta\delta]} \left|x_i(t) - x_{i}\left(t_{i}\right)\right| + |x_{N+1}(\beta\delta) - x_{\mathfrak{f}, N+1}| \\ 
						& < \varepsilon,
					\end{aligned}
				\end{equation}
				where
				\[
					t_0 \coloneq  0, \quad t_1 \coloneq \beta_1\delta, \quad t_2 \coloneq \sum_{i=1}^{2} \beta_i \delta, \quad \dots, \quad t_N \coloneq \sum_{i=1}^{N} \beta_i \delta = \beta \delta.
				\]
				Under the assumptions on $x_{\mathfrak{in}}$, $\ell$, and $\mathcal{B}_1, \dots, \mathcal{B}_{N+1}$ in \Cref{corollary:approx1cont}, for each $\varepsilon > 0$ there exists $\delta_0 > 0$ such that for all $\delta \in (0, \delta_0)$ there is a continuous map $\prod_{i=1}^{N+1}{\mathcal{B}}_i\longrightarrow  C^0([0,1];\mathbb{R}^d)\times\mathcal{J}_{\ell}^N$ assigning $x_{\mathfrak{f}} \mapsto (\zeta,\beta_1,\dots,\beta_N)$ with \eqref{equation:adi}.
		\end{corollary}
		\begin{proof}
			Without loss of generality, it holds $x_{\mathfrak{f},i} \neq x_{{\mathfrak{in}},j}$ for all $1\leq i\neq j \leq N$, as otherwise one could suitably perturb some targets.
			The assertions thus follow from \Cref{corollary:approx1} and \Cref{corollary:approx1cont} (up to relabeling of particles) by successively driving for each $1\leq i \leq N$ the $i$-th particle during $[t_{i-1}, t_i]$ to its target, while keeping the others approximately at their current location during that interval. In this process, \Cref{corollary:approx1} (or \Cref{corollary:approx1cont}) provide $N$ versions of $\beta$ and $\delta_0$, say $\beta_1, \dots, \beta_N > 0$ and $\delta_{0, 1},\dots,\delta_{0, N}$, and we fix $\delta_0 \in (0, \min\{\delta_{0,1}, \dots, \delta_{0,N}\})$. The additional continuity statement follows from \Cref{corollary:approx1cont} by gluing the path of the control particle via \eqref{equations:approx_cor}  in a differentiable way at the points $t_1, \dots, t_{N-1}$.
			Through \eqref{equation:zetaformula_}, this first provides a control $\zeta \in C^0([0,\beta\delta]; \mathbb{R}^d)$ which can be extended (e.g., constantly) to $[\beta \delta,1]$. By~\Cref{theorem:pert} and a density argument,~$\zeta$ can be approximated by $C^{\infty}([0,\beta\delta]; \mathbb{R}^d)$-functions (or the extension of $\zeta$ by $C^{\infty}([0,1]; \mathbb{R}^d)$-functions) that maintain \eqref{equation:adi}.
		\end{proof}

		We can now conclude \Cref{theorem:maintrajectorial}, which provides global approximate controllability of all particles in any time and simultaneously allows to prescribe the path along which the first~$N$ particles are steered to the target. 
		
		\begin{proof} [Proof of \Cref{theorem:maintrajectorial}]
		Let $M \in \mathbb{N}$ and $U$ be a neighborhood of $\Delta_{N}$ satisfying $\mathfrak{c}([0,T_{\mathfrak{f}}]) \cap \overline{U} = \emptyset$. Moreover, fix numbers
		\[
			0 = a_1 < a_2 < b_1 < a_3 < b_2 < \dots < a_M < b_{M-1} < b_M = T_{\mathfrak{f}}
		\]
		such that there exists an open covering of~$\mathfrak{c}([0, T_{\mathfrak{f}}])$ by balls $(\mathcal{O}_m)_{m\in\{1,\dots,M\}} \subset \mathcal{M}^{N}\setminus \overline{U}$ with $\mathfrak{c}([a_m, b_m]) \subset \mathcal{O}_m$ and $\mathcal{O}_m$ having diameter less than $\varepsilon/(2\sqrt{N})$ for all $1 \leq m \leq M$. 
		
		The proof will be concluded by choosing $\zeta$ through the following steps in a way that 
		\[
			(x_1(t), \dots, x_N(t)) \in \mathcal{O}_m, \quad |x_{N+1}(T_{\mathfrak{f}}) - x_{\mathfrak{f},N+1}| < \varepsilon/2
		\]
		for all $t \in [a_m, b_m]$ and $1 \leq m \leq M$.
		
		 {\it Step 1. Stability property.} By local well-posedness and a compactness argument, given any $r_0 > 0$ and $R_0 > \max_{t \in [0,T_{\mathfrak{f}}]}|\mathfrak{c}(t)| +2r_0$, there are neighborhoods $\mathcal{U}_m \subset \mathcal{O}_m$ of $\mathfrak{c}([a_m, b_m])$, $1 \leq m \leq M$ and numbers $s_1,\dots,s_M \in (0,1)$ with the following property. If $(\widetilde{x}_{\mathfrak{in},1}, \dots, \widetilde{x}_{\mathfrak{in},N}) \in \mathcal{U}_m$ for some $1\leq m \leq M$ and $\widetilde{x}_{\mathfrak{in},N+1} \in B_{\mathcal{M}}(0, R_0)$ with $|\widetilde{x}_{\mathfrak{in},N+1}-\widetilde{x}_{\mathfrak{in},i}|  > r_0$ for all $1\leq i \leq N$, then the solution $\widetilde{x}$ to the uncontrolled problem
		\begin{equation}\label{equation:uncont}
			\begin{cases}
				\dot{\widetilde{x}}_i = K(\widetilde{x}_i-\widetilde{x}_{N+1}) + F_i(\widetilde{x}),\\
				\dot{\widetilde{x}}_{N+1} = F_{N+1}(\widetilde{x}),\\
				\widetilde{x}(0) = (\widetilde{x}_{\mathfrak{in},1}, \dots, \widetilde{x}_{\mathfrak{in},N}, \widetilde{x}_{\mathfrak{in},N+1}), 
			\end{cases}\boxnote{\qquad \quad \, (1 \leq i \leq N)}
		\end{equation}
		satisfies
		\begin{gather*}
			\forall t \in [0, s_m]\colon (\widetilde{x}_1(t), \dots, \widetilde{x}_N(t)) \in \mathcal{O}_m.
		\end{gather*}

		{\it Step 2. Construction of the control.} To define $\zeta$ on $[0, b_1]$, \Cref{corollary:approx2} is employed to drive $(x_1, \dots, x_N)$ in a time $t_1 \in (0, b_1)$ into a neighborhood of the target $\mathfrak{c}(b_1)$ that is contained in $\mathcal{U}_1\cap\mathcal{U}_2$, while ensuring that $(x_1(t), \dots, x_N(t)) \in \mathcal{U}_1$ for all $t \in [0,t_1]$. This is possible thanks to \eqref{equation:adi}, which allows to push $(x_1, \dots, x_N)$ approximately along piecewise linear paths contained in $\mathcal{U}_1$. In addition, \Cref{corollary:approx2} allows to ensure $x_{N+1}(t_1) \in B_{\mathcal{M}}(0, R_0)$ and $|x_{N+1}(t_1)- x_i(t_1)| > r_0$ for all $1 \leq i \leq N$. Then, the control is switched off during $[t_1, \min\{b_1, t_1+s_1\}]$. If $t_1+s_1 < b_1$, the control is activated at time $t_1+s_1$ to drive $(x_1, \dots, x_N)$ similarly as before back to a neighborhood of $\mathfrak{c}(b_1)$ contained in $\mathcal{U}_1\cap \mathcal{U}_2$ quicker than a time $t_2 \in (0, b_1-t_1 -s_1)$, and all without leaving $\mathcal{O}_1$. This idea can be repeated a finite number of times until $(x_1(b_1), \dots, x_N(b_1))$ belongs to a neighborhood of $\mathfrak{c}(b_1)$ contained in $\mathcal{U}_1\cap\mathcal{U}_2$, while ensuring that $(x_1(t), \dots, x_N(t)) \in \mathcal{O}_1$ for all $t \in [0,b_1]$. 
		
		Similarly, $\zeta$ is constructed on $(b_1, b_2]$ by driving $(x_1, \dots, x_N)$  in a time $t_2 > 0$ shorter than $b_2-b_1$ into a neighborhood of $\mathfrak{c}(b_2)$ that is contained in $\mathcal{U}_2\cap\mathcal{U}_3$, while preventing $(x_1, \dots, x_N)$ from leaving $\mathcal{U}_2$ during this stage. Simultaneously, one can ensure $x_{N+1}(b_1 + t_2) \in B_{\mathcal{M}}(0, R_0)$ and $|x_{N+1}(b_1+t_2)- x_{i}(b_1+t_2)| > r_0$ for all $1 \leq i \leq N$. Then, the controls are switched off on $[b_1+t_2, \min\{b_2, b_1+t_2 +s_2\}]$, and if $b_1+t_2 +s_2 < b_2$ the state $(x_1, \dots, x_N)$ is driven in a time shorter than $b_2 - b_1 - t_2 - s_2$ back to a neighborhood of $\mathfrak{c}(b_2)$ contained in $\mathcal{U}_2\cap\mathcal{U}_3$ without leaving $\mathcal{O}_2$. 
		This idea can be iterated until $(x_1(b_2), \dots, x_N(b_2))$ belongs to a neighborhood of $\mathfrak{c}(b_2)$ that is contained in $\mathcal{U}_2\cap\mathcal{U}_3$, while ensuring that $(x_1(t), \dots, x_N(t)) \in \mathcal{O}_2$ for all $t \in [b_1,b_2]$. In this way, one can successively define a suitable control~$\zeta$ also on $(b_2, b_3], \dots, (b_{M-1}, T_{\mathfrak{f}}]$. Thanks to \Cref{corollary:approx2}, it can be simultaneously achieved that $|x_{N+1}(T_{\mathfrak{f}}) - x_{\mathfrak{f},N+1}| < \varepsilon/2$. Finally, $\zeta \in C^{\infty}_c((0,T_{\mathfrak{f}}); \mathbb{R}^d)$ is obtained by redefining the previous choice of $\zeta$ through \Cref{theorem:pert} by using a density argument.
		\end{proof}
		
		The next corollary establishes UAC in the sense of \Cref{definition:uac} (with $m=(N+1)d$ and $l = d$) from any initial state $x_{\mathfrak{in}}$ to a product of closed balls in arbitrary time. As controls that vanish at the endpoints are convenient, e.g., for the continuous gluing of iterated control stages, we work for  $T > 0$ with the space
		\[
			C^0_0([0,T];\mathbb{R}^d) \coloneq \{ \zeta \in C^0([0,T];\mathbb{R}^d) \, | \, \zeta(0) = \zeta(T) = 0 \},
		\]
		which is complete with the supremum norm. 
			
		\begin{corollary}\label{corollary:uacallparticles}
			Let $T_{\mathfrak{f}} > 0$ and $x_{\mathfrak{in}} \in \Delta_{N+1}^{\complement}$ be fixed. There are closed balls $\mathcal{B}_1, \dots,\mathcal{B}_{N+1} \subset \mathcal{M}$ with $\mathcal{B} = \mathcal{B}_1\times\dots\times\mathcal{B}_{N+1}\subset \Delta_{N+1}^{\complement}$ such that for any $\varepsilon > 0$ there is a continuous map
			\begin{equation}\label{equation:contmapphi}
				\Psi_{\varepsilon} \colon {\mathcal{B}} \longrightarrow  C^0_0([0,T_{\mathfrak{f}}];\mathbb{R}^d)
			\end{equation}
			that assigns to each $x_{\mathfrak{f}}  \in \mathcal{B}$ a control $\zeta = \Psi_{\varepsilon} (x_{\mathfrak{f}})$ for which the corresponding solution $x = (x_1, \dots, x_{N+1})$ to \eqref{equation:particlesystem_deterministic} is defined on $[0, T_{\mathfrak{f}}]$ and satisfies
			\begin{equation}\label{equation:act}
			\sum_{i = 1}^{N+1} |x_{i}(T_{\mathfrak{f}}) - x_{\mathfrak{f},i}| < \varepsilon.
			\end{equation} 
		\end{corollary}
		\begin{proof}
			For sufficiently small $\ell > 0$, which depends on the choice of $x_{\mathfrak{in}} \in \Delta_{N+1}^{\complement}$, let $\mathcal{B}_1, \dots, \mathcal{B}_{N+1}$ be provided by \Cref{corollary:approx2}. Also, fix any $\varepsilon > 0$ and without loss of generality assume $\varepsilon \ll \operatorname{dist}(\mathcal{B},\Delta_{N+1})$, which ensures below that approximations of targets in $\mathcal{B}$ remain at a uniform distance to the diagonal. 
			By local well-posedness and compactness of ${\mathcal{B}}$, there exists a time $s_0 \in (0, T_{\mathfrak{f}})$ such that for all $x_{\mathfrak{f}} \in {\mathcal{B}}$ the solutions to the uncontrolled problem \eqref{equation:uncont} with initial data from an $(\varepsilon/3)$-neighborhood of $x_{\mathfrak{f}}$ cannot exit the $(\varepsilon/2)$-neighborhood of $x_{\mathfrak{f}}$ in time $s_0$.

			To define the map $\Psi_{\varepsilon} $, let $x_{\mathfrak{f}} \in {\mathcal{B}}$ be arbitrary, set $\mathfrak{c}(t) \coloneq (x_{\mathfrak{in},1}, \dots, x_{\mathfrak{in},N})$ for all $t\geq 0$, and apply \Cref{theorem:maintrajectorial} on the time interval $[0, T_{\mathfrak{f}}-s_0]$ with target $(\mathfrak{c}(T_{\mathfrak{f}}-s_0), x_{\mathfrak{in},N+1}) = x_{\mathfrak{in}}$ to steer the system into a small neighborhood~$U_{\mathfrak{in}}$ of $x_{\mathfrak{in}}$ with $\operatorname{dist}(U_{\mathfrak{in}}, \mathcal{B}) \in (\ell, 2\ell)$, using (by density \& \Cref{theorem:pert}) a control that vanishes at $t = 0$ and $t = T_{\mathfrak{f}}-s_0$. By the choice of $U_{\mathfrak{in}}$, \Cref{corollary:approx2} can be applied starting from any state in $U_{\mathfrak{in}}$ with target in $\mathcal{B}$ as fixed above. Moreover, as the choice of~$s_0$ is uniform with respect to~$x_{\mathfrak{f}}$ from $\mathcal{B}$, this control can be taken independently of~$x_{\mathfrak{f}}$. Finally, \Cref{corollary:approx2} provides a control that drives the system from $x(T_{\mathfrak{f}}-s_0)$ to the $\varepsilon/3$-neighborhood of $x_{\mathfrak{f}}$ in a time $s \in (0,s_0)$ that is independent of~$x_{\mathfrak{f}}$ from~$\mathcal{B}$. In particular, \Cref{corollary:approx2} ensures that this control depends continuously on such~$x_{\mathfrak{f}}$. On $[T_{\mathfrak{f}}-s_0+s, T_{\mathfrak{f}}]$, set the control zero. 
			
			To avoid that the above determined control jumps at $t = T_{\mathfrak{f}}-s_0$ and $t = T_{\mathfrak{f}}-s_0+s$, one can multiply it with a smooth cutoff that has values
			\[
				\begin{cases}
					1, & \mbox{ if } t \in [0, T_{\mathfrak{f}}-s_0-\widetilde{s}] \cup [T_{\mathfrak{f}}-s_0+\widetilde{s}, T_{\mathfrak{f}}-s_0+s-\widetilde{s}],\\
					0, & \mbox{ if } t \in [T_{\mathfrak{f}}-s_0-\widetilde{s}/2, T_{\mathfrak{f}}-s_0+\widetilde{s}/2] \cup [T_{\mathfrak{f}}-s_0+s-\widetilde{s}/2, T_{\mathfrak{f}}]
				\end{cases}
			\]
			for sufficiently small $\widetilde{s} \in (0,s/3)$ such that \Cref{theorem:pert} ensures \eqref{equation:act} for the corresponding perturbed problem. As $\widetilde{s}$ can be fixed independently of $x_{\mathfrak{f}} \in \mathcal{B}$, the proof is complete.
		\end{proof}

		\subsection{Exact and solid controllability}
		We are now in the position to complete the proof of our main controllability result.
		
		\begin{proof}[Proof of \Cref{theorem:maincontroltheorem}]
		The statement on solid controllability follows by combining \Cref{corollary:uacallparticles} with \Cref{proposition:ucisc}. In particular, if $(\Psi_{\varepsilon})_{\varepsilon > 0}$ and $\mathcal{B}$ are obtained via \Cref{corollary:uacallparticles}, then for $R > 0$ and $y \in \Delta_{N+1}^{\complement}$ with $B_{\mathcal{M}^{N+1}}(y, R) \subset \mathcal{B}$ and any $r \in (0, R)$ one can take the compact set $\mathcal{C} = \Psi_{r/2}(\overline{B}_{\mathcal{M}^{N+1}}(y, R))$ in \Cref{definition:solidcontrollability}.
		
		Global exact controllability in time $T_{\mathfrak{f}} > 0$ can be shown as follows. First, solid controllability from arbitrary $x_{\mathfrak{in}} \in \Delta_{N+1}^{\complement}$ in time $T_{\mathfrak{f}}/2$ with the choice \smash{$\Phi(\widehat{\zeta}) \coloneq x(T_{\mathfrak{f}}/2; x_{\mathfrak{in}}, \widehat{\zeta})$} for \smash{$\widehat{\zeta} \in \mathcal{C}$} in \Cref{definition:solidcontrollability} implies exact controllability in time $T_{\mathfrak{f}}/2$ from $x_{\mathfrak{in}} \in \Delta_{N+1}^{\complement}$ to any target $\widehat{x}_{\mathfrak{f}} \in \overline{B}_{\mathcal{M}^{N+1}}(y, R-r)$.  Now, one can argue by time-reversibility: given any target $x_{\mathfrak{f}} \in \Delta_{N+1}^{\complement}$, apply solid controllability with the target $\widehat{x}_{\mathfrak{f}} = \overline{x}(T_{\mathfrak{f}}/2)$, where $\overline{x}$ solves the reversed problem
		\begin{equation}
			\label{equation:particlesystem_deterministic_reversed}
			\begin{cases}
				\dot{\overline{x}}_i = -K(\overline{x}_i-\overline{x}_{N+1}) - F_i(\overline{x}), \\
				\dot{\overline{x}}_{N+1} = -F_{N+1}(\overline{x}) + \overline{\zeta},\\
				\overline{x}(0) = x_{{\mathfrak{f}}}
			\end{cases}
		\end{equation}
		with $\overline{\zeta}$ chosen via \Cref{theorem:maintrajectorial} such that $\overline{x}(T_{\mathfrak{f}}/2) \in B_{\mathcal{M}^{N+1}}(y, R-r)$. Here, we used that \eqref{equation:particlesystem_deterministic_reversed} is of the same form as \eqref{equation:particlesystem_deterministic}, thus all previously obtained controllability statements hold for it, as well. Finally, by taking the control
		\[
			\zeta(t) \coloneq \begin{cases}
				\widehat{\zeta}(t) & \mbox{ if } t \in [0, T_{\mathfrak{f}}/2),\\
				- \overline{\zeta}(T_{\mathfrak{f}}-t) & \mbox{ if } t \in [T_{\mathfrak{f}}/2, T_{\mathfrak{f}}],
			\end{cases}
		\]
		it follows that $x(T_{\mathfrak{f}}; x_{\mathfrak{in}}, \zeta) = x_{\mathfrak{f}}$. As $\widehat{\zeta}$ and~$\overline{\zeta}$ vanish at $T_{\mathfrak{f}}/2$, the control $\zeta$ defined above is continuous.
	\end{proof}

		\section{Exponential mixing for stochastic systems}\label{section:expmixing}
		
		\Cref{theorem:maincontroltheorem} together with an abstract criterion from \cite{MertzNersesyanRissel2024} yield exponential mixing for particle systems driven by degenerate stochastic noise
		\begin{equation}\label{equation:S1}
			\begin{cases}
				\dot{x}_i  = K_i(x_i-x_{N+1}) + F_i(x), & i \in \{1,\dots,N\}, \\ 
				\dot{x}_{N+1}  = F_{N+1}(x) + \zeta,\\
				x(0) = x_{\mathfrak{in}} \in \Delta^{\complement}_{N+1},
			\end{cases}
		\end{equation}
		where $\zeta$ is a general decomposable noise as introduced below and the interactions
        \[
            K_1,\dots,K_N\colon \mathcal{M}\setminus \{0\} \longrightarrow \mathbb{R}^d, \quad F_1, \dots, F_{N+1}\colon \Delta^{\complement}_{N+1} \to \mathbb{R}^d
        \]
        are of the same form as in \eqref{equation:particlesystem_deterministic}, but additionally assumed to be smooth on their respective domains.

Since the right-hand side in \eqref{equation:S1} is of the form treated in \Cref{section:proofofmainresult}, the deterministic controllability results of that section apply directly: by \Cref{theorem:maincontroltheorem}, the system \eqref{equation:S1} is exactly controllable and solidly controllable from any point $x_{\mathfrak{in}} \in \Delta^{\complement}_{N+1}$ in any time.

		 Throughout, let $\mathcal{P}(\Delta^{\complement}_{N+1})$ denote the Borel probability measures on $\Delta^{\complement}_{N+1}$, endowed with the total variation distance  
		\begin{equation}\label{equation:var}
			\begin{aligned}
				\|\mu_1-\mu_2\|_{\operatorname{var}} \coloneq  
				\sup_{ \Gamma\in\mathcal{B}(\Delta^{\complement}_{N+1})} |\mu_1(\Gamma)-\mu_2(\Gamma)|,
			\end{aligned}
		\end{equation}
		 where  $\mathcal{B}(\Delta^{\complement}_{N+1})$ is the Borel~$\sigma$-algebra on $\Delta^{\complement}_{N+1}$.  Moreover, let $L^{\infty}(\Delta^{\complement}_{N+1})$ be the space of bounded measurable functions $f\colon \Delta^{\complement}_{N+1}\longrightarrow \mathbb{R}$, equipped with the supremum norm $\|f\|_\infty \coloneq \sup_{x\in \Delta^{\complement}_{N+1}} |f(x)|$.

		\subsection{Decomposable noise} \label{subsection:decomposablenois}

		Given $L^2((0,1);\mathbb{R}^d)$-valued i.i.d. random variables $(\eta_k)_{k \in \mathbb{N}}$, decomposable noise refers to a random process of the form
		\begin{equation}\label{noise}
			\zeta(t) = \sum_{k=1}^{\infty} \mathbb{I}_{[k-1,k)}(t)\eta_k(t+1-k),
		\end{equation}
		provided that the family $(\eta_k)_{k \in \mathbb{N}}$ satisfies the following assumption.
		
		\begin{description}[leftmargin=0pt,labelindent=0pt]
			\item[\hypertarget{D}{(D)} Decomposability.] 
			{\itshape  The random variables $\{\eta_k\}_{k \in \mathbb{N}}$ are independent copies of a random variable of the form $\sum_{j=1}^\infty b_j\xi_j e_j$, where
				\begin{itemize}
					\item $b_j > 0$ for all $j\in \mathbb{N}$ and $\sum_{j=1}^\infty b_j^2 < \infty$,
					\item  $\{\xi_{j}\}_{j\in\mathbb{N}}$ 
					are independent scalar random variables, each having the same positive continuous density $\rho$ with respect to the Lebesgue measure such that $\int_{-\infty}^{+\infty} s^2 \rho(s)d s<\infty$,
					\item $\{e_j\}_{j\in\mathbb{N}}$ is an orthonormal basis in~$L^2((0,1);\mathbb{R}^d)$.
			\end{itemize}}
		\end{description} In particular, this assumption implies $\mathbb{E}\|\eta_1\|_{L^2}^2 < \infty$, and the support of the law of $\eta_1$ is all of $L^2((0,1);\mathbb{R}^d)$.
        
We impose a Lyapunov assumption on the system \eqref{equation:S1}.
		\begin{description}[leftmargin=0pt,labelindent=0pt]
			\item[\hypertarget{Ld}{(L)} Lyapunov assumption.] 
			{\itshape There are $\alpha > 0$ and $C > 0$, and a $C^1$-function $V\colon \Delta^{\complement}_{N+1}\longrightarrow [1,\infty)$ with compact sublevel set~$\{V \leq R\}$ for each~$R \geq 1$ such that the following statement holds. Given any $\eta\in L^2((0,1);\mathbb{R}^d)$, $x_{\mathfrak{in}} \in \Delta^{\complement}_{N+1}$, and solution $x$ of \eqref{equation:S1} with $\zeta=\eta$ and $x(0) = x_{\mathfrak{in}}$, one has
			\begin{equation}\label{equation:deterministic_Lyapunov}
					\frac{d}{d t} V(x(t)) \leq -\alpha V(x(t)) + C\left(1+|\eta(t)|^2\right)
			\end{equation}
			for almost all $t \in [0,1]$.}
		\end{description}

In particular, Gr\"onwall's inequality applied to \eqref{equation:deterministic_Lyapunov}, together with the compactness of the sublevel set $\{V \leq R\}$ for~$R \geq 1$, ensures that for every $\eta \in L^2((0,1);\mathbb{R}^d)$ and $x_{\mathfrak{in}} \in \Delta^{\complement}_{N+1}$ the corresponding solution of \eqref{equation:S1} with $\zeta = \eta$ exists on all of $[0,1]$ and remains in $\Delta^{\complement}_{N+1}$.

Under the above assumptions, the trajectories of \eqref{equation:S1} with $\zeta$ given by \eqref{noise},  restricted to integer times $k$, that is $x_k=x(k)$, define a Markov family $\{x_k,\mathbb{P}_x\}_{k\ge 1}$ with transition probabilities $\{P_k(x,\Gamma)\}_{k\ge 1}$ (see~\cite[Section~1.3]{KS-12}). The associated Markov semigroups are given by
		\begin{gather*}
			\mathfrak{P}_k \colon L^{\infty}(\Delta^{\complement}_{N+1})  \longrightarrow L^{\infty}(\Delta^{\complement}_{N+1}), \quad \mathfrak{P}_k f(x) \coloneq \int f(y)\, P_k(x, {d} y), \\
			\mathfrak{P}_k^* \colon \mathcal{P}(\Delta^{\complement}_{N+1}) \longrightarrow \mathcal{P}(\Delta^{\complement}_{N+1}), \quad \mathfrak{P}_k^* \mu(\Gamma) \coloneq \int P_k(x,\Gamma)\, \mu(d x)
		\end{gather*}
		for $k \geq 1$, and we recall that a probability measure $\mu \in \mathcal{P}(\Delta^{\complement}_{N+1})$ is stationary if $\mathfrak{P}_1^*\mu = \mu$.
		
		\begin{theorem}\label{theorem:main_decomposable}
			 Under assumptions  \hyperlink{D}{\rm{(D)}}  and \hyperlink{Ld}{\rm{(L)}}, the family $\{x_k, \mathbb{P}_x\}_{k \ge 1}$ admits a unique stationary measure $\mu \in \mathcal{P}(\Delta^{\complement}_{N+1})$ that is exponentially mixing. More precisely, there exist constants $\gamma > 0$ and $C>0$ such that
			\begin{equation*}
				\|\mathfrak{P}_k^* \lambda -\mu\|_{\operatorname{var}} \leq C\operatorname{e}^{-\gamma k}  \langle V, \lambda \rangle_{\Delta^{\complement}_{N+1}} \boxnote{\qquad \qquad \qquad \, (k \in \mathbb{N})}
			\end{equation*}
			for each $\lambda\in \mathcal{P}(\Delta^{\complement}_{N+1})$ with
			\[
				\langle V, \lambda \rangle_{\Delta^{\complement}_{N+1}} = \int_{\Delta^{\complement}_{N+1}} V(x) \lambda(d x)< \infty.
			\]
		\end{theorem}
		\begin{proof}Let us set  $E \coloneq L^2((0,1);\mathbb{R}^d)$ and consider the map
\[
		S\colon \Delta^{\complement}_{N+1} \times E \longrightarrow \Delta^{\complement}_{N+1}, \quad (x_{\mathfrak{in}},\eta) \mapsto x(1), 
	\]	
    where $x(t)$ denotes the solution of \eqref{equation:S1} with initial condition $x_{\mathfrak{in}} \in \Delta^{\complement}_{N+1}$ and forcing $\zeta = \eta\in E$. Owing to \eqref{equation:deterministic_Lyapunov} and the smoothness of $K_i, F_i$, this map is globally well-defined and infinitely Fr\'echet differentiable. The sequence $(x_k)_{k \ge 1}$, defined by~$x_k = x(k)$ for $k \ge 1$, satisfies the relation $x_k=S(x_{k-1},\eta_k)$.

	To complete the proof, we verify Conditions~1--4 of the abstract criterion in Theorem~2.8 of~\cite{MertzNersesyanRissel2024}.  As required by this reference, we fix any compatible Riemannian metric on $\Delta_{N+1}^{\complement}$.

	{\it Condition~1 (Lyapunov structure).} 
 	Multiplying the inequality \eqref{equation:deterministic_Lyapunov} by~$\operatorname{e}^{\alpha t}$, integrating over~$[0,1]$, and taking the expectation yields the required estimate $\mathbb{E}_x V(x_1) \leq q V(x) + A$ with $q = \operatorname{e}^{-\alpha } \in (0,1)$ and $A =  {C}(1 + \mathbb{E}\|\eta_1\|_E^2)$.

  	{\it Condition~2 (Approximate controllability).} Let us fix any $p \in \Delta^{\complement}_{N+1} $. The global exact controllability of \Cref{theorem:maincontroltheorem} implies that any initial state can be steered into an arbitrarily small neighborhood of~$p$ in time $1$, which is the approximate controllability required by~\cite{MertzNersesyanRissel2024}.

	{\it Condition~3 (Solid controllability).} This condition is composed of a regularity part and a solid controllability part, both at the point~$p$ fixed above. The regularity requirement is satisfied because $S$ is infinitely differentiable, and solid controllability from $p$ in time $1$ follows via~\Cref{theorem:maincontroltheorem}. 

	{\it Condition~4 (Decomposability).} The law of $\eta_k$ satisfies this condition thanks to assumption~\hyperlink{D}{\rm{(D)}}.

	In conclusion, Theorem~2.8 of~\cite{MertzNersesyanRissel2024} yields the existence and uniqueness of the stationary measure $\mu \in \mathcal{P}(\Delta^{\complement}_{N+1})$ and its exponential mixing.
		\end{proof}

		\subsection{Riesz-type interacting particle systems}\label{subsection:Riesz}
		
		The previous results are illustrated for an explicit class of interacting particle systems with Riesz-type potentials. We focus only on this representative example, as it covers all spatial dimensions $d \geq 2$,  but similar results can be obtained likewise for other systems such as those listed in \Cref{subsubsection:examples}, including the point vortex system.

Consider a system of $N+1$ particles that are located at time $t \geq 0$ at the positions $x(t) = (x_1(t),\dots,x_{N+1}(t)) \in \Delta^{\complement}_{N+1}$, which evolve according to the gradient flow of the Riesz interaction energy
\begin{equation}\label{equation:RieszEnergyApplication}
	\mathcal{H}_{N+1}(x_1,\dots,x_{N+1}) = \sum_{1 \leq i < j \leq {N+1}} a_{ij}g(x_i - x_j), \quad g(x) \coloneq \frac{c_s}{|x|^{s}},
\end{equation}
where $s \in (0,d)$, $d\ge2$, and $c_s>0$ is a normalization constant.  
For simplicity, it is assumed here that all interaction intensities are equal to one, that is $a_{ij} = 1$ for $1\leq i < j \leq N+1$, so that the energy $\mathcal{H}_{N+1}$ is purely repulsive with $\mathcal{H}_{N+1} > 0$ on $\Delta^{\complement}_{N+1}$. Moreover, after introducing additional damping and forcing (see \eqref{eq:damped_ODE_riesz}) one can verify \hyperlink{Ld}{\rm{(L)}} as discussed next. 

Concerning global well-posedness and Lyapunov structure, we first note that the identity
\begin{equation}\label{equation:e}
	\sum_{i=1}^{N+1} \langle x_i, \nabla_{x_i} \mathcal{H}_{N+1}(x)\rangle = -s\,\mathcal{H}_{N+1}(x)
	\boxnote{\qquad  \qquad (x \in \Delta^{\complement}_{N+1})}
\end{equation}
holds because $g$ and $\mathcal{H}_{N+1}$ are positively homogeneous of degree $-s$. Now, the deterministic motions of undamped Riesz particles with energy $\mathcal{H}_{N+1}$ are governed by  
\begin{equation}\label{equation:ODE_riesz}
	\dot x_i(t) = -\nabla_{x_i} \mathcal{H}_{N+1}(x_1(t),\dots,x_{N+1}(t)),
\end{equation}
where
\begin{equation}\label{equation:gradient}
	\nabla_{x_i} \mathcal{H}_{N+1}(x) = - s\,c_s \sum_{1 \leq j \neq i \leq N+1}  \frac{x_i - x_j}{|x_i - x_j|^{s+2}}
\end{equation}
for all $1 \leq i \leq N+1$. Thus, along any solution of \eqref{equation:ODE_riesz}, the interaction energy is non-increasing:
\[
	\frac{d}{dt} \mathcal{H}_{N+1}(x(t)) = - \sum_{i=1}^{N+1} \left|\nabla_{x_i} 	\mathcal{H}_{N+1}(x(t))\right|^2 \leq 0.
\]
Moreover, as $g$ is repulsive and blows up at the origin, one has $\mathcal{H}_{N+1}(x) \to \infty$ as $x \to \Delta_{N+1}$. Hence, the boundedness of the energy keeps the minimal interparticle distance strictly positive on any finite time interval. Further, by using~\eqref{equation:e} one has
\begin{equation}\label{equation:atmostlineargrowth}
	\begin{aligned}
		\frac{d}{dt} \sum_{i=1}^{N+1} |x_i(t)|^2
		&= -2 \sum_{i=1}^{N+1} \langle x_i(t), \nabla_{x_i} \mathcal{H}_{N+1}(x(t))\rangle
		= 2s\,\mathcal{H}_{N+1}(x(t)) \\&\leq 2s\,\mathcal{H}_{N+1}(x(0)),
	\end{aligned}
\end{equation}
which shows that $\sum_i |x_i(t)|^2$ grows at most linearly in $t$. Since the minimal interparticle distance remains bounded below on every finite time interval, the gradient \eqref{equation:gradient} and the vector field in \eqref{equation:ODE_riesz} stay uniformly bounded there. Combined with \eqref{equation:atmostlineargrowth}, this prevents a finite-time blow-up and yields global well-posedness of~\eqref{equation:ODE_riesz}.

The following lemma is used to derive a Lyapunov estimate.  For the sake of completeness, a proof is briefly given in Appendix~\ref{section:Riesz}.
\begin{lemma}\label{lemma:coercivity}
	There is a constant $c_* > 0$, depending only on $N, d, s, c_s$, such that
	\begin{equation}\label{equation:coercivity}
		\sum_{i=1}^{N+1} \left|\nabla_{x_i}\mathcal{H}_{N+1}(x)\right|^2 \geq c_*\,\mathcal{H}_{N+1}(x)^{\frac{2s+2}{s}}
	\end{equation}
	for all $x \in \Delta^{\complement}_{N+1}$.
\end{lemma}

To construct a Lyapunov function satisfying the assumption \hyperlink{Ld}{\rm{(L)}}, consider the confining potential
\begin{equation}\label{equation:confiningpotential}
	Q(x_1,\dots,x_{N+1}) \coloneq \sum_{i=1}^{N+1} |x_i|^2,
\end{equation}
and define the total energy
\[
	V(x) \coloneq 1 + \mathcal{H}_{N+1}(x) + Q(x).
\]
Since $\mathcal{H}_{N+1}(x) \to \infty$ as $x \to \Delta_{N+1}$ and $Q(x) \to \infty$ as $|x| \to \infty$, the function $V \colon \Delta^{\complement}_{N+1} \longrightarrow [1,\infty)$ has compact sublevel sets in $\Delta^{\complement}_{N+1}$. 
Along solutions to the damped and forced dynamics
\begin{equation}\label{eq:damped_ODE_riesz}
	\dot x(t) = -\nabla \mathcal{H}_{N+1}(x(t)) - x(t) + \mathbb{B}\eta(t),
\end{equation}
where $\mathbb{B}\eta$ denotes the forcing acting on the $x_{N+1}$-component, the identity \eqref{equation:e} gives
\begin{equation}\label{eq:dotV}
	\frac{d}{dt} V(x(t))
	= - \Sigma(x(t))
	+ 3s\,\mathcal{H}_{N+1}(x(t)) - 2\,Q(x(t))
	+ \langle \nabla_{x_{N+1}} V(x(t)), \eta(t)\rangle,
\end{equation}
where $\Sigma \coloneq \sum_{i=1}^{N+1}|\nabla_{x_i}\mathcal{H}_{N+1}|^2$. For the last term on the right-hand side, the~bound
\[
	|\nabla_{x_{N+1}}V|^2 = |\nabla_{x_{N+1}}\mathcal{H}_{N+1} + 2x_{N+1}|^2 \leq 2\Sigma + 8Q
\]
together with Young's inequality gives
\[\langle \nabla_{x_{N+1}} V, \eta\rangle \leq \frac18\Sigma + \frac12 Q + 4|\eta|^2,
\]
so that \eqref{eq:dotV} becomes
\[
	\frac{d}{dt} V(x(t)) \leq -\frac78\Sigma(x(t)) + 3s\,\mathcal{H}_{N+1}(x(t)) - \frac32 Q(x(t)) + 4|\eta(t)|^2.
\]
Since $s > 0$ and $(2s+2)/s > 1$, the function $h \mapsto (3s+\alpha)h - \frac78 c_* h^{(2s+2)/s}$ is bounded above on $[0,\infty)$ for every $\alpha \in (0,\frac32]$ by some finite $A_0 = A_0(\alpha)$; hence \Cref{lemma:coercivity} yields
\[
	-\frac78\Sigma + 3s\,\mathcal{H}_{N+1} \leq -\frac78 c_*\,\mathcal{H}_{N+1}^{(2s+2)/s} + 3s\,\mathcal{H}_{N+1} \leq -\alpha\,\mathcal{H}_{N+1} + A_0.
\]
Combining this with $-\frac32 Q \leq -\alpha Q$ and fixing such an $\alpha$, there is a constant $C > 0$ with
\[
	\frac{d}{dt} V(x(t)) \leq -\alpha V(x(t)) + C\left(1+|\eta(t)|^2\right),
\]
which is precisely~\hyperlink{Ld}{\rm{(L)}}. Thus, by \Cref{theorem:main_decomposable}, exponential mixing holds for the Riesz-interacting particle system \eqref{eq:damped_ODE_riesz}.

		\section*{Acknowledgments}
		
		MR is grateful for the support by a start-up grant from ShanghaiTech University. 
		LM is thankful for support through NSFC Grant No.~12271364 and GRF Grant No.~11302823. VN and MR are partially supported by NSFC No.~12571156.
				
		\appendix

            \addcontentsline{toc}{section}{Appendix}
\addtocontents{toc}{\protect\setcounter{tocdepth}{-1}}

		\section{Appendix}
		\subsection{Perturbative result}\label{section:perturbative}
		The following known perturbative result can be found, for instance, in \cite[Theorem~4.4]{Gamkrelidze1978}.

		\begin{lemma}\label{theorem:pert}
			Let $m \in \mathbb{N}$, $O \subset \mathbb{R}^m$ open, $G \subset O$ open and bounded with $\overline{G} \subset O$, and $t_1 < t_2$. For ${x}_{{\tau}} \in G$ and ${\tau} \in (t_1,t_2)$, assume that ${x}\colon [t_1,t_2]\longrightarrow G$ solves
			\[
				\dot{{x}}(t) = {\mathscr{F}}(t,{x}(t)), \quad {x}({\tau}) = {x}_{{\tau}} \boxnote{\qquad \qquad  \quad (t \in [t_1,t_2]),}
			\]
			where ${\mathscr{F}}$ is continuous on $[t_1,t_2]\times O$ and $x \mapsto {\mathscr{F}}(t,x)$ is Lipschitz continuous on $O$ with  Lipschitz constant independent of $t$. Moreover, let $\widehat{\mathscr{F}}$ be another continuous function on $[t_1,t_2]\times O$ with $x \mapsto \widehat{\mathscr{F}}(t,x)$ Lipschitz continuous on~$O$ uniformly in~$t$. Then, given any $\varepsilon > 0$, there exists $\delta > 0$ such that
			\[
				|\widehat{\tau} - {\tau}| + |{x}_{{\tau}} - \widehat{x}_{\widehat{\tau}}| + \max\limits_{\substack{s,t \in [t_1, t_2],\\x \in \overline{G}}}\left|\int_s^t({\mathscr{F}} - \widehat{\mathscr{F}})(r, x) \, dr \right| < \delta
			\]
			implies that the equation
			\[
				\dot{\widehat{x}} = \widehat{\mathscr{F}}(t,\widehat{x}), \quad \widehat{x}(\widehat{\tau}) = \widehat{x}_{\widehat{\tau}}  \boxnote{\qquad \qquad \qquad \, (t \in [t_1,t_2]),}
			\]
			has a unique solution on $[t_1,t_2]$ with $\max_{t \in [t_1,t_2]}  |{x}(t) - \widehat{x}(t)|  \leq \varepsilon$.
		\end{lemma}

		\subsection{Solid controllability}\label{section:appendix}
		
		Let $m \in \mathbb{N}$ and denote by $\overline{B}_{\mathbb{R}^m}(z, R)$ the closed ball of radius $R>0$ in $\mathbb{R}^m$ with center $z \in \mathbb{R}^m$. The following result is Proposition 1.1 in \cite{Sh-2007}. 
		\begin{proposition} \label{P:1}
			Let $R > 0$, $\varepsilon \in (0,R)$, $z \in \mathbb{R}^m$, and $S\colon \overline{B}_{\mathbb{R}^m}(z,R) \to \mathbb{R}^m$ be a continuous map such that for all $x \in \overline{B}_{\mathbb{R}^m}(z,R)$ it holds $|S(x) - x| \leq \varepsilon$. Then $\overline{B}_{\mathbb{R}^m}(z, R - \varepsilon) \subset S \left(\overline{B}_{\mathbb{R}^m}(z,R)\right)$.
		\end{proposition}

		Let $D \subset \mathbb{R}^m$ be closed, $\mathscr{F}\colon \mathbb{R}^m\setminus D \longrightarrow \mathbb{R}^m$ Lipschitz continuous on~$\mathbb{R}^m\setminus U$ for any neighborhood $U$ of $D$, $l \in \mathbb{N}$, $\mathscr{B}\colon \mathbb{R}^l \longrightarrow \mathbb{R}^m$ a bounded linear operator, $T>0$, and $x_0 \in \mathbb{R}^m\setminus D$. Further, let $X$ be a Banach space with continuous embedding $X\hookrightarrow L^1([0,T]; \mathbb{R}^l)$. Finally, denote by $\mathcal{Q}(T,x_0)$ the set of controls
		$\zeta \in X$ for which the problem 
		\begin{equation}\label{equation:S}
			\dot{x}(t) = \mathscr{F}(x(t)) + \mathscr{B}\zeta(t), \quad x(0) = x_0
		\end{equation}
		admits a unique solution $t \mapsto x(t; x_0, \zeta)$ on $[0,T]$.  We equip $\mathcal{Q}(T,x_0)$ with the topology induced by $X$.
		
		\begin{definition}\label{definition:solidcontrollabilityappendix}
			The system \eqref{equation:S} is solidly controllable from $x_0$ in time $T$, if there exist $\varepsilon>0$, a ball $B$, and a compact set $\mathcal{C} \subset \mathcal{Q}(T,x_0)$ such that for any continuous map $\Phi\colon \mathcal{C} \to \mathbb{R}^m$ satisfying
			\begin{equation*}\label{equation:epscondsolid}
				\sup_{\zeta \in \mathcal{C}} | \Phi(\zeta) - x(T;x_0, \zeta) |  \leq  \varepsilon,
			\end{equation*}
			one has $B \subset \Phi(\mathcal{C})$.
		\end{definition}
		
		\begin{definition}\label{definition:uac}
			The system \eqref{equation:S} is uniformly approximately controllable (UAC) from $x_0$ to a nonempty closed $\mathcal{K} \subset \mathbb{R}^m\setminus D$ in time $T$ with a family $(\Psi_{\delta})_{\delta > 0} \subset C^0(\mathcal{K}; \mathcal{Q}(T,x_0))$,
			if 
			\[
				\sup\limits_{\overline{x} \in \mathcal{K}} | x(T; x_0, \Psi_{\delta}(\overline{x}))-\overline{x}| \leq \delta
			\]
			for each $\delta > 0$.
		\end{definition}
		
		\begin{proposition}\label{proposition:ucisc}
			If \eqref{equation:S} is UAC from $x_0$ to $\overline{B}_{\mathbb{R}^m}(z,R)$ in time $T$ with a family $(\Psi_{\delta})_{\delta > 0} \subset C^0(\overline{B}_{\mathbb{R}^m}(z,R); \mathcal{Q}(T,x_0))$, then \eqref{equation:S} is solidly controllable from $x_0$ in time $T$ with the choices $\varepsilon\in(0,R)$, $B = \overline{B}_{\mathbb{R}^m}(z,R-\varepsilon)$, and 	$\mathcal{C} \coloneq \Psi_{\varepsilon/2}(\overline{B}_{\mathbb{R}^m}(z,R))$ in \Cref{definition:solidcontrollabilityappendix}.
		\end{proposition}
		
		\begin{proof}
			Let $\varepsilon\in(0,R)$. By UAC, there exists a continuous map
			\[
			\Psi = \Psi_{\varepsilon/2}\colon \overline{B}_{\mathbb{R}^m}(z, R) \to \mathcal{Q}(T,x_0)
			\]
			such that
			\begin{equation} \label{equation:e2}
				\sup\limits_{\overline{x} \in \overline{B}_{\mathbb{R}^m}(z, R)} 
				| x(T; x_0, \Psi(\overline{x}))-\overline{x}| \leq \varepsilon/2.
			\end{equation}
			Let $\mathcal{C} \coloneq \Psi(\overline{B}_{\mathbb{R}^m}(z,R))$, which is compact in $\mathcal{Q}(T,x_0)$ as a continuous image of a compact set. Now, let $\Phi\colon \mathcal{C} \to \mathbb{R}^m$ be any continuous mapping such that
			\begin{equation} \label{equation:e3}
				\sup_{\zeta \in \mathcal{C}} 
				|\Phi(\zeta) - x(T;x_0,\zeta) | \leq \varepsilon/2. {}
			\end{equation}
			Consider the composite map
			\[
			\Phi \circ \Psi \colon \overline{B}_{\mathbb{R}^m}(z,R) \to \mathbb{R}^m.
			\]
			Using \eqref{equation:e2} and \eqref{equation:e3}, we find that $|\Phi(\Psi(x))-x| \leq \varepsilon$ for all $x \in \overline{B}_{\mathbb{R}^m}(z,R)$. Applying \Cref{P:1} gives 
			\[
			\overline{B}_{\mathbb{R}^m}(z,R-\varepsilon) \subset (\Phi \circ \Psi)\left(\overline{B}_{\mathbb{R}^m}(z,R)\right)
			= \Phi(\mathcal{C}).
			\]
			Thus, the system \eqref{equation:S} is solidly controllable from $x_0$ in time $T$.
		\end{proof}
		
		\subsection{Curves}\label{section:curve}
		Let $m \in \mathbb{N}$ and denote for $k \geq 0$ and open $J\subset \mathbb{R}$ the $k$-th derivative of a function $f \in C^k(\overline{J}; \mathbb{R}^m)$ by $f^{(k)}$. We recall a classical method of constructing curves confined to a given ball and with prescribed behavior at the endpoints.
		\begin{lemma}\label{lemma:localcurve}
			Let $T>0$, $M \in \mathbb{N}$, and $x, y \in B \coloneq B_{\mathbb{R}^m}(z,R)$ for $z \in \mathbb{R}^m$ and $R > 0$. Further, fix any $u_1, \dots, u_M,w_1,\dots, w_M \in \mathbb{R}^m$.
			There exists a curve $\chi \in C^{\infty}([0,T]; \mathbb{R}^m)$ with $\chi([0,T]) \subset B$ such that
			\begin{equation*}
				\quad \chi(0) = x, \quad \chi(T) = y, \quad 
				\chi^{(k)}(0) = u_k, \quad \chi^{(k)}(T) = w_k
			\end{equation*}
			for all $1 \leq k \leq M$ and
			\begin{equation*}
				\chi^{(k)}(0) = 0, \quad \chi^{(k)}(T) = 0
			\end{equation*}
			for all $k > M$. In addition, given any open bounded set $W \subset \mathbb{R}^m$, $s \in \mathbb{N}$, and $r \in (0,R)$, one can assign $(x, y, u_1, \dots, u_M,w_1,\dots, w_M)\mapsto\chi$ through a smooth map 
			\[
				B_{\mathbb{R}^m}(z,r)^2\times W^{2M} \longrightarrow C^s([0,T]; \mathbb{R}^m).
			\]
		\end{lemma}
		
		\begin{proof}
			Since $B$ is convex, $S\coloneq \{(1-\xi)x+\xi y \, | \,  0 \leq \xi \leq 1\} \subset B$. Moreover, there exists $\delta>0$ with $\operatorname{dist}(S,\mathbb{R}^m\setminus B) > \delta$. Now, take $\varphi \in C^{\infty}([0,\infty);[0,1])$ such that $\varphi(s) = 1$  for $0 \leq s \leq 1/3$ and $\varphi(s) = 0$ for $s \geq 2/3$, and a function $h \in C^{\infty}([0,T]; [0,1])$ such that $h(t)=0$ for $0 \leq t \leq T/4$ and $h(t) = 1$ for $3T/4 \leq t \leq T$.
			Then, define $\gamma(t) \coloneq x+h(t)(y-x)$ for $t \in [0,T]$. In particular, it holds
			\[
			\gamma([0,T]) \subset S \subset B, \quad \gamma(0) = x, \quad \gamma(T)=y, \quad \gamma^{(j)}(0) = 0, \quad \gamma^{(j)}(T)=0
			\]
			for $j \in \mathbb{N}$. Now, for small $\varepsilon \in (0,T/4)$ fixed later, consider the smooth paths 
			\[
			A_k(t) \coloneq \frac{t^k}{k!}\varphi \left( \frac{t}{\varepsilon} \right), \quad B_k(t) \coloneq \frac{(-1)^k (T-t)^k}{k!}\varphi \left( \frac{T-t}{\varepsilon} \right)
			\]
			for $1 \leq k \leq M$ and $t\in[0,T]$ that will be used to perturb $\gamma$ near its endpoints on $[0,T/4]\cup[3T/4,T]$. By construction,
			\[
			A_k^{(j)}(T) = B_k^{(j)}(0) = 0, \quad  A_k^{(j)}(0) = B_k^{(j)}(T) = \begin{cases}
				1, & \mbox{ if } j = k,\\
				0, & \mbox{ if } j \neq k
			\end{cases}
			\]
			for all $0 \leq j \leq M$ and $1 \leq k \leq M$. Finally, define the function
			\begin{equation}\label{equation:appendixphi}
				\chi(t) \coloneq \gamma(t)+\sum_{k=1}^M u_k A_k(t)+\sum_{k=1}^M w_k B_k(t)
			\end{equation}
			for $t \in [0,T]$. To ensure that $\chi([0,T]) \subset B$, fix sufficiently small $\varepsilon \in (0,T/4)$ uniformly with respect to $x, y \in B$ and $u_1, \dots, u_M,w_1,\dots, w_M \in W$ such that
			\[
			\sup_{0\leq t\leq T} |\chi(t)-\gamma(t)| \leq \sum_{k=1}^M \frac{\varepsilon^k}{k!}\left(|u_k|+|w_k|\right) < \delta.
			\]
			The claim on smooth dependence follows because $0 < \delta < R - r$ for $(x,y) \in B_{\mathbb{R}^m}(z,r)^2$ and $\chi$ defined via \eqref{equation:appendixphi} is affine in $x, y, u_1, \dots, u_M,w_1,\dots, w_M$.
		\end{proof}

\subsection{Riesz interaction energy}\label{section:Riesz}
In this section, we prove the coercivity estimate stated in~\Cref{lemma:coercivity} for the purely repulsive Riesz energy
\begin{equation}\label{equation:rea}
	\mathcal{H}_{N+1}(x_1,\dots,x_{N+1}) = \sum_{1 \leq i < j \leq {N+1}} \frac{c_s}{|x_i - x_j|^{s}} \boxnote{\quad \quad (x \in \Delta^{\complement}_{N+1}),}
\end{equation}
where $s \in (0,d)$, $d\ge2$, and $c_s>0$.

\begin{proof}[Proof of \Cref{lemma:coercivity}]
	 For each $x \in \Delta^{\complement}_{N+1}$, we denote $r(x) \coloneq \min_{i < j}|x_i - x_j|$ for the associated minimal interparticle distance and further denote the ratio
	 \[
	 	J(x) \coloneq \frac{\Sigma(x)}{\mathcal{H}_{N+1}(x)^{\frac{2s+2}{s}}},
	 \]
	 where
	\begin{gather*}
		\Sigma(x) \coloneq \sum_{i=1}^{N+1}|\nabla_{x_i}\mathcal{H}_{N+1}(x)|^2.
	\end{gather*}
	As $J$ is dilation invariant and $\mathcal{H}_{N+1}(x) \in [c_s,2^{-1}c_sN(N+1)]$ for all $x \in \Delta^{\complement}_{N+1}$ with $r(x) = 1$, the proof can be completed by showing that
	\begin{equation}\label{equation:sigmab}
		\sigma_* \coloneq \inf_{\substack{x \in \Delta^{\complement}_{N+1}, \\ r(x) = 1}}\Sigma(x) > 0.
	\end{equation}
	Note that the identity (see \eqref{equation:e})
	\begin{equation}\label{equation:ee}
		\sum_{i=1}^{N+1} x_i \cdot \nabla_{x_i} \mathcal{H}_{N+1}(x) = -s\,\mathcal{H}_{N+1}(x)
		\boxnote{\qquad  \quad \,\,\, (x \in \Delta^{\complement}_{N+1})}
	\end{equation}
	already implies that $\Sigma(x) > 0$ holds for each $x \in \Delta^{\complement}_{N+1}$. The uniform lower bound \eqref{equation:sigmab} can now be verified by contradiction. Suppose that $\sigma_* = 0$ and let $(x^{n})_{n\in\mathbb{N}} \subset \{r=1\}$ be a sequence with $\lim_{n\to \infty}\Sigma(x^{n}) = 0$. For each~$n$, at least two particles realize the minimal distance. Thus, after passing to a subsequence, renaming indices, and applying a translation (noting that $\Sigma$ and $r$ are translation invariant), it can be assumed that there is a number $R > 0$ and a particle cluster indexed by a set $\mathcal{G} \subset \{1,\dots, N+1\}$ such that 
	\[
	|\mathcal{G}| \geq 2, \quad \bigcup_{i \in \mathcal{G}, \, n \in \mathbb{N}} \{x_i^n\} \subset B_{\mathbb{R}^d}(0,R), \quad \inf_{\substack{i \in \mathcal{G}, \,  j \notin \mathcal{G}}} \liminf_{n\to\infty} |x_i^n - x_j^n| = \infty.
	\]
	Passing to another subsequence, for each $i \in \mathcal{G}$ there exists $\overline{x}_i \in \mathcal{M}$ with
	\begin{gather*}
		\lim_{n\to \infty} (x_i^n)_{i \in \mathcal{G}} = (\overline{x}_i)_{i \in \mathcal{G}}, \quad \inf_{\substack{i \in \mathcal{G}, \,  j \notin \mathcal{G}}} \liminf_{n\to\infty} |\overline{x}_i - x_j^n| = \infty,\\ \min_{i, j \in \mathcal{G}, \, i \neq j}|\overline{x}_i - \overline{x}_j| = 1.
	\end{gather*}
	However, denoting by $\mathcal{H}_{|\mathcal{G}|}$ the Riesz energy of the particles indexed by $\mathcal{G}$, the assumption $\Sigma(x^n) \to 0$ implies
	\[
	\sum_{i\in\mathcal{G}} \left|\nabla_{x_i}\mathcal{H}_{|\mathcal{G}|}((\overline{x}_i)_{i \in \mathcal{G}})\right| = 0,
	\]
	contradicting the version of \eqref{equation:ee} which holds for $\mathcal{H}_{|\mathcal{G}|}$ due to positive homogeneity of the kernel $x \mapsto c_s |x|^{-s}$ in \eqref{equation:rea} with $N = |\mathcal{G}|-1$.
\end{proof}

		\bibliographystyle{alpha}
		\bibliography{main}
		
	\end{document}